\documentclass{article}
\usepackage[utf8]{inputenc}
\usepackage{amsmath}
\usepackage{amssymb}
\usepackage{tikz}
\usepackage{graphicx}
\usepackage{float}
\usepackage{framed}
\usepackage[hang,flushmargin]{footmisc}
\usepackage{comment}
\usepackage{ytableau}
\usepackage[nottoc,numbib]{tocbibind}
\usepackage{mathtools}
\usepackage{amsthm}
\newtheorem{theorem}{Theorem}[section]
\newtheorem{corollary}[theorem]{Corollary}
\newtheorem{lemma}[theorem]{Lemma}
\newtheorem{definition}[theorem]{Definition}
\newtheorem{proposition}[theorem]{Proposition}
\newtheorem{example}[theorem]{Example}
\newtheorem{remark}[theorem]{Remark}

\usepackage{scalerel,stackengine}
\stackMath
\newcommand\reallywidehat[1]{%
\savestack{\tmpbox}{\stretchto{%
  \scaleto{%
    \scalerel*[\widthof{\ensuremath{#1}}]{\kern-.6pt\bigwedge\kern-.6pt}\documentclass{article}
\usepackage[utf8]{inputenc}
\usepackage{amsmath}
\usepackage{amssymb}
\usepackage{tikz}
\usepackage{graphicx}
\usepackage{float}
\usepackage{framed}
\usepackage[hang,flushmargin]{footmisc}
\usepackage{comment}
\usepackage{ytableau}
\usepackage[nottoc,numbib]{tocbibind}
\usepackage{mathtools}
\usepackage{amsthm}
\newtheorem{theorem}{Theorem}[section]
\newtheorem{corollary}[theorem]{Corollary}
\newtheorem{lemma}[theorem]{Lemma}
\newtheorem{definition}[theorem]{Definition}
\newtheorem{proposition}[theorem]{Proposition}

\newtheorem{remark}[theorem]{Remark}

\usepackage{scalerel,stackengine}
\stackMath
\newcommand\reallywidehat[1]{%
\savestack{\tmpbox}{\stretchto{%
  \scaleto{%
    \scalerel*[\widthof{\ensuremath{#1}}]{\kern-.6pt\bigwedge\kern-.6pt}%
    {\rule[-\textheight/2]{1ex}{\textheight}}
  }{\textheight}%
}{0.5ex}}%
\stackon[1pt]{#1}{\tmpbox}%
}
\parskip 1ex

\title{Limit formulas for the trace of the functional calculus of quantum channels for $SU(1,1)$}       
\author{robinvhaastrecht }
\date{August 2024}

\newcommand{\nm}[1]{ || #1 || }
\newcommand{\C}{\mathbb{C}}
\newcommand{\R}{\mathbb{R}}
\newcommand{\N}{\mathbb{N}}
\newcommand{\Hc}{\mathcal{H}}
\newcommand{\li}[1]{\overline{ #1}}
\newcommand{\Tr}{\mathrm{Tr}}
\newcommand{\D}{\mathbb{D}}

\begin{document}

    {\rule[-\textheight/2]{1ex}{\textheight}}
  }{\textheight}%
}{0.5ex}}%
\stackon[1pt]{#1}{\tmpbox}%
}
\title{Functional calculus of quantum channels for the holomorphic discrete series of $SU(1,1)$}
\author{Robin van Haastrecht}
\date{August 2024}

\newcommand{\nm}[1]{ || #1 || }
\newcommand{\C}{\mathbb{C}}
\newcommand{\R}{\mathbb{R}}
\newcommand{\N}{\mathbb{N}}
\newcommand{\Hc}{\mathcal{H}}
\newcommand{\li}[1]{\overline{ #1}}
\newcommand{\Tr}{\mathrm{Tr}}
\newcommand{\D}{\mathbb{D}}

\begin{document}

\maketitle

\begin{abstract}
The tensor product of two holomorphic discrete series representations of $SU(1,1)$ can be decomposed as a direct sum of infinitely many discrete series. I shall introduce equivariant quantum channels for each component of the direct sum, mapping bounded operators on one factor of the tensor product to operators on the component. Next I prove a limit formula for the trace of the functional calculus and I prove that the limit can be expressed using generalized Husimi functions or using Berezin transforms.
\end{abstract}

\section{Introduction}

Wehrl-type inequalities are of fundamental interest in Mathematical Physics \cite{wehrl}. First Lieb proved a Wehrl-type $L^p$-inequality for the Heisenberg group in 1978 \cite{lieb}, and many years later in 2014 Lieb and Solovej proved a Werhl-type inequality for the group $SU(2)$ \cite{liebsolBloch}. Afterwards, they formulated and proved Wehrl-type inequalities for some representations of the group $SU(n)$ \cite{liebsolSymm} and the discrete series representations of $SU(1,1)$ \cite{kuli} (see also Frank's paper \cite{frank} for a unified approach). The $L^p$-inequalities were later generalized by Zhang for even integers $p$ of holomorphic discrete series of Hermitian Lie groups \cite{zhangCon}. See also \cite{frank, sugi}. In order to prove the Wehrl inequality for $SU(2)$, Lieb and Solovej studied the trace of the functional calculus of certain quantum channels and proved it converges to an integral of matrix coefficients. Later, they showed a similar limit holds for certain quantum channels of $SU(1,1)$ \cite{liebsolSU11}. This was generalized by Zhang \cite{zhangCon} for the scalar holomorphic discrete series for Hermitian Lie groups.

The quantum channels studied by Lieb and Solovej are defined using the leading component in the tensor product representations in $SU(2)$ \cite{liebsolBloch}. In my previous paper \cite{haas} I studied a new family of quantum channels by considering all the irreducible components of the tensor product decomposition and I calculated the limit of the trace of the functional calculus of these $SU(2)$-equivariant quantum channels. Similar results were proven independently in \cite{soletal} where, among other things, this limit was obtained in terms of generalized Husimi functions. In those papers $SU(2)$-equivariant quantum channels were considered for operators on finite-dimensional spaces. In the current paper, we will consider $SU(1,1)$-equivariant channels. A natural candidate to define these channels on is the holomorphic discrete series, as they have $L^2$-matrix coefficients. We briefly explain our construction below.

Let $\Hc_{\mu}$ be the holomorphic discrete series representations with lowest weight $\mu$ and consider the tensor product decomposition of two irreducible representations of $SU(1,1)$ \cite{repk}

\begin{equation}
\label{decompintro}
\Hc_{\mu} \otimes \Hc_{\nu} \cong \bigoplus_{k=0}^{\infty} \Hc_{\mu + \nu + 2k}.
\end{equation}
Earlier \cite{liebsolSU11, zhangCon} the quantum channel
$$ \mathcal{T}^{\nu}_{\mu}: B(\Hc_{\mu}) \rightarrow B(\Hc_{\mu + \nu})$$
defined by considering the leading component
$$ \mathcal{T}^{\nu}_{\mu}(A) = P (A \otimes I_{\nu}) P^* \in B(\Hc_{\mu + \nu}).$$
was already considered. Here $I_{\nu}$ is the identity operator on $\Hc_{\nu}$ and
$$P: \Hc_{\mu} \otimes \Hc_{\nu} \rightarrow \Hc_{\mu + \nu}$$
is an intertwining partial isometry. It was proved that for any integer $n$
$$ \lim_{\nu \rightarrow \infty} \frac{1}{\nu} \Tr(\mathcal{T}^{\nu}(u \otimes u^*)^n) = \int_{SU(1,1)} |\langle g \cdot u, u_0 \rangle|^{2n} dg.$$
Here $u_0 \in \Hc_{\mu}$ is the function constantly equal to one. In this paper we introduce general quantum channels, defined by projecting onto the irreducible $k$ component of our decomposition (\ref{decompintro}). We define

$$ \mathcal{T}_{\mu,k}^{\nu}(A) = P_k ( A \otimes I_{\nu} ) P_k^*.$$
We prove this map is trace-preserving up to a constant and completely positive and study the limit formula of the trace of the functional calculus. It will turn out that the Berezin transform and the Toeplitz operator will be useful to study the limit. Then we introduce generalized Husimi functions and the operator $E_{\mu,k}$, and we obtain the following Theorem.

\begin{theorem}
\label{main}
Let $\psi \in C^1([0,1]) $, $\psi(0) = 0$ and $A \in B(\Hc_{\mu})$ positive and of trace one. Then
$$ \lim_{\nu \rightarrow \infty} \frac{1}{\nu} \Tr(\psi(\mathcal{T}_{\mu,k}^{\nu}(A))) = \int_{\D} \psi(H_{\mu}^k(A)(z)) d \iota (z).$$
Furthermore, let $f \in L^1(\D)$ such that $R_{\mu}^*(f)$ is positive and of trace one. Then
$$ \lim_{\nu \rightarrow \infty} \frac{1}{\nu} \Tr(\psi(\mathcal{T}_{\mu,k}^{\nu}(R_{\mu}^*(f)))) = \int_{\D} \psi(E_{\mu,k}(f)(z)) d \iota (z).$$
\end{theorem}
Note that the condition that $\psi \in C^1([0,1])$ and $\psi(0) = 0$ ensures our operator is trace-class. Here we use the notation
$$ E_{\mu,k}(f) \coloneqq \frac{(\mu)_{k}}{k!} \sum_{j=0}^k (-1)^j \binom{k}{j} B_{\mu + j}(f),$$
where $B_{\mu+j}$ is the Berezin transform from Definition \ref{berezindef} and the generalized Husimi function $H_{\mu}^k$ from Definition \ref{genHus}.

We use the following strategy. First we calculate
$$R_{\nu + \mu - 2k} \mathcal{T}_{\mu,k}^{\nu}(R_{\mu}^*(f)),$$
where $R_{\nu}$ is from Definition \ref{Rmudef}. We then notice that
$$\mathcal{T}_{\mu,k}^{\nu} R_{\mu}^* = R_{\mu + \nu - 2k}^* B_{\mu + \nu - 2k}^{-1} R_{\mu + \nu - 2k} \mathcal{T}_{\mu,k}^{\nu} R_{\mu}^*,$$
and that we can calculate
$$ \lim_{\nu \rightarrow \infty} \frac{1}{\nu} \Tr(( (\nu-1)R_{\nu}^*(f))^{2n}) = \int_{\D} f(z)^{2n} d \iota(z).$$
for any integer $n$ using the spectral theory of the Berezin transform $B_{\nu}$. Then we prove
$$\lim_{\nu \rightarrow \infty} \nm{(\nu B_{\nu})^{-1}B_{\nu_0}(f) - B_{\nu_0}(f)}_2 = 0$$
and obtain Theorem \ref{main} for $\psi(x) = x^{2n}$ for any $n$. We extend our result to odd integers $n$ by writing $\psi(x) = x^n$ as a converging sum of polynomials of even integers. Finally we use denseness of polynomials in $C([0,1])$ to prove Theorem \ref{main} in full generality.

We note that the Berezin transform has been studied extensively as it is closely related to quantization on Kähler manifolds in Geometry and Mathematical Physics \cite{alieng, BMSCMP, untup}. Some of our results about Berezin transforms might be obtained from these results, but we prove more precise results using representations of $SU(1,1)$.

The paper is organized as follows. In Section \ref{defpre} we introduce the discrete series representations of $SU(1,1)$ as reproducing kernel Hilbert spaces and recall some relevant known results. In Section \ref{rnuber} we define and analyse the Toeplitz operator, the covariant symbol and the Berezin transform, all of which are $SU(1,1)$-invariant. We also go through the Plancherel theory for the symmetric space $\D = SU(1,1)/U(1)$ and analyse our operators in that context. In Section \ref{channelingsec} we define our quantum channels and prove some of their basic properties. In Section \ref{tracfor} we study trace formulas for the functional calculus. In Section \ref{inversebersec} we analyse the inverse of the Berezin transform. Finally, in Section \ref{tracecalc} we prove Theorem \ref{main}, the main theorem.

\section*{Acknowledgements}
I want to thank Genkai Zhang for inspiring discussions and useful suggestions for handling the Berezin transform.

\section{Definitions and preliminaries}
\label{defpre}

Let $SU(1,1) = \{ \begin{pmatrix} a & b\\ \li{b} & \li{a} \end{pmatrix} \mid a,b \in \C \ \mathrm{and} \ |a|^2 - |b|^2 = 1 \}$ and $\D = \{ z \in \C \mid |z| < 1 \}$ the open unit disk.

\begin{definition}
For $\nu > 1$ we let
$$\Hc_{\nu} \coloneqq \{ f \colon \D \rightarrow \C \ \mathrm{analytic} \mid \int_{\D} |f(z)|^2 (1 - |z|^2)^{\nu} \frac{dz}{\pi (1 - |z|^2)^2} < \infty \}.$$
This is a Hilbert space with norm
$$ \nm{f}_{\nu}^2 \coloneqq (\nu - 1) \int_{\D} |f(z)|^2(1 - |z|^2)^{\nu} \frac{dz}{\pi (1 - |z|^2)^2}.$$
\end{definition}

Note that the norm is normalized such that $\nm{1}_{\nu}=1$. For convenience we write

$$ d \iota_{\nu}(z) \coloneqq (1 - |z|^2)^{\nu} \frac{(\nu - 1) dz}{\pi (1 - |z|^2)^{2}},$$
and
$$ d \iota(z) \coloneqq d \iota_0(z).$$

This space will give rise to a representation of $SU(1,1)$ when $\nu$ is an integer and a projective representation otherwise (as in the complex case we have to make a choice of root when $\nu$ is not an integer). For an integer $\nu > 1$ this will be the discrete series of $SU(1,1)$, as described in \cite[chapter 2]{knapp}. If $g = \begin{pmatrix} a & b \\ \li{b} & \li{a} \end{pmatrix}$, the action is

$$(g \cdot f)(z) = (-b z + \li{a})^{- \nu} f( \frac{a z - \li{b}}{-b z + \li{a}})$$
and gives a unitary representation of $SU(1,1)$. In what follows we let $\nu > 1$ be an integer. We now make a remark on the space $\D$.

\begin{remark}
\label{repsu11facts}
Note that $SU(1,1)$ acts on the symmetric space $SU(1,1) / U(1)$ by left multiplication. This space is isomorphic to $\D$ when the action on $\D$ is given by $g \cdot z = \frac{az - \li{b}}{-bz + \li{a}}$ for $g = \begin{pmatrix} a & b \\ \li{b} & \li{a} \end{pmatrix}$. Then we see that $U(1)$ is exactly the subgroup fixing $0$, and $SU(1,1)$ is transitive as it is transitive on the unit ball. Thus $\D \cong SU(1,1) / U(1)$. We note that there is a $SU(1,1)$-invariant metric on $\D$ making it a Riemannian symmetric space, see \cite[chapter VI, Theorem 1.1]{helgDS}.
\end{remark}

We calculate the norm of the monomials in our space.

\begin{lemma}
In the space $\Hc_{\nu}$
$$\nm{z^j}^2_{\nu} = \frac{j!}{ (\nu)_j}.$$
Furthermore, the functions $z^j$ and $z^k$ are orthogonal when $j \neq k$.
\end{lemma}

\begin{proof}
Direct computation.
\end{proof}

Note that $\Hc_{\nu}$ is a space of holomorphic functions, so the span of $\{ \sqrt{ \frac{(\nu)_i}{i!} } z^i \}_{i=0}^{\infty}$ is dense and we conclude that it is an orthonormal basis for the Hilbert space $\Hc_{\nu}$. Writing
$$ (1 - x \li{y})^{- \nu} = \sum_{i = 0}^{\infty} \frac{ (\nu)_i }{i!} (x \li{y})^i,$$
it directly follows that $\Hc_{\nu}$ is a reproducing kernel space with kernel $K^{\nu}(x,y) = (1 - x \li{y})^{-\nu}$. 

We now study the tensor product of the representations. If we see the tensor product as holomorphic functions on the product of two discs, which are integrable, then we have the decomposition \cite{repk}

\begin{equation}
\label{TensorProdDec}
\Hc_{\mu} \otimes \Hc_{\nu} = \bigoplus_{k = 0}^{\infty} \Hc_{\mu + \nu + 2k}.
\end{equation}
Thus there exists a projection

$$P_k \colon \Hc_{\mu} \otimes \Hc_{\nu} \rightarrow \Hc_{\mu + \nu + 2k}.$$
These projections have been studied previously and it was found that they have the following form \cite{peet, pezh, zag}.

\begin{proposition}
The map $P_k \colon \Hc_{\mu} \otimes \Hc_{\nu} \rightarrow \Hc_{\mu + \nu + 2k}$ is given by

\begin{flalign*}
& P_k(F)(\zeta) = C_{\mu, \nu, k} \sum_{j=0}^k (-1)^j \binom{k}{j} \frac{1}{(\mu)_{j} (\nu)_{k - j}} \partial^j_z \partial^{k-j}_w F |_{z=w=\zeta}
\\ & = C_{\mu, \nu, k} \int_{\D^2} F(z,w) ( \frac{\li{z}}{1 - \zeta \li{z}} - \frac{\li{w}}{1 - \zeta \li{w}})^k (1 - \zeta \li{z})^{- \mu} (1 - \zeta \li{w})^{- \nu} \cdot
\\ & d \iota_{\mu} (z) d \iota_{\nu}(w).
\end{flalign*}
Furthermore, the injective isometry $P_k^*$ is given by

\begin{flalign*}
& P_k^*(f)(z,w)
\\ & = C_{\mu,\nu,k} \int_{\D} (1 - z \li{\zeta})^{- \mu - k} (1 - w \li{\zeta})^{- \nu - k} f( \zeta) (z - w)^k d \iota_{\mu + \nu + 2k}(\zeta)
\\ & = C_{\mu,\nu,k} \int_{\D} (1 - z \li{\zeta})^{- \mu} (1 - w \li{\zeta})^{- \nu} f( \zeta) (\frac{z}{1 - z \li{\zeta}} - \frac{w}{1 - w \li{\zeta}})^k d \iota_{\mu + \nu + 2k}(\zeta).
\end{flalign*}
\end{proposition}

\begin{proof}
Follows from \cite{peet, pezh} and calculations like in \cite{haas}. 


\end{proof}

We want to deduce the value of the constant $C_{\mu,\nu,k}$. First we recall the following summation formula \cite[Corollary 2.2.3]{andaskroy}.

\begin{lemma}
\label{Gausshgf}
For $n$ a non-negative integers and $b,c$ integers such that $|c| \geq n$ we have that
$${}_2F_1(-n, b, c, 1) = \frac{(c-b)_n}{(c)_n}$$
\end{lemma}

Now we deduce the value of the constant $C_{\mu,\nu,k}$.

\begin{proposition}
\label{constantcmunuk}
We have
$$ C_{\mu, \nu, k}^{-2} = \frac{k! (\mu + \nu + k + 1)_k}{(\mu)_k (\nu)_k}.$$
\end{proposition}

\begin{proof}
We know that
$$P_k P_k^*\colon \Hc_{\mu + \nu + 2k} \rightarrow \Hc_{\mu + \nu + 2k}$$
is the identity. We now note that
\begin{flalign*}
& P_k^*(1)(z,w) = C_{\mu,\nu,k} \int_{\D} (1 - z \li{\zeta})^{-\mu-k} (1 - w \li{\zeta})^{-\nu-k} (z-w)^k d \iota_{\mu+\nu+2k}(\zeta)
\\ & = (z-w)^k.
\end{flalign*}
Thus we see

\begin{flalign*}
& 1 = \langle 1, 1 \rangle_{\mu + \nu + 2k} = \langle P_k P_k^*(1),1 \rangle_{\mu + \nu + 2k} = \langle P_k^*(1), P_k^*(1) \rangle_{\Hc_{\mu} \otimes \Hc_{\nu}}
\\ & = C^2_{\mu, \nu, k} \nm{(z-w)^k}^2_{\Hc_{\mu} \otimes \Hc_{\nu}}.
\end{flalign*}
We calculate
\begin{flalign*}
& C_{\mu, \nu, k}^{-2} = \nm{(z-w)^k}_{\Hc_{\mu} \otimes \Hc_{\nu}}^2 = \sum_{j=0}^k \binom{k}{j}^2 \nm{z^j w^{k-j}}^2_{\Hc_{\mu} \otimes \Hc_{\nu}}
\\ & = \sum_{j=0}^k \binom{k}{j}^2 \frac{j! (k-j)!}{(\mu)_j (\nu)_{k-j}} = k! \sum_{j=0}^k \binom{k}{j} \frac{1}{(\mu)_j (\nu)_{k-j}}
\\ & = \frac{k!}{(\nu)_k} \sum_{j=0}^k (-1)^j \frac{(-k)_j}{j!} \frac{(-1)^j (-\nu - k + 1)_j}{(\mu)_j} = \frac{k!}{(\nu)_k} {}_2F_1(-k, -\nu-k+1, \mu;1)
\\ & = \frac{k! (\mu + \nu + k + 1)_k}{(\mu)_k (\nu)_k}.
\end{flalign*}
\end{proof}

\section{The Toeplitz operator, covariant symbol and Berezin transform}
\label{rnuber}

We define some other $SU(1,1)$-invariant operators on the representation
spaces. First we define $R_{\nu}$, commonly called the covariant symbol \cite{alieng, engl, zhangbound}.

\begin{definition}
\label{Rmudef}
The covariant symbol map $R_{\nu}$,
$$R_{\nu}\colon B(\Hc_{\nu}) \rightarrow C^{\infty}(\D)$$
is given by $R_{\nu}(A)(z) = A(z,z)(1 - |z|^2)^{\nu}$.
\end{definition}

We prove some facts for this map.

\begin{proposition}
\label{factsRmu}
The maps
$$R_{\nu}\colon S^p(\Hc_{\nu}) \rightarrow L^p(\D, d \iota)$$
for $1 \leq p < \infty$ and the map
$$R_{\nu}\colon B(\Hc_{\nu}) \rightarrow L^{\infty}(\D, d \iota)$$
are well-defined, continuous with respect to the natural norm, and injective.
\end{proposition}

\begin{proof}
Let $A \in B(\Hc_{\nu})$ with kernel $A(x,y)$. We note that the function $z \mapsto R_{\nu}(A)(z) = A(z,z) (1-|z|^2)^{\nu}$ is holomorphic in the first coordinate and antiholomorphic in the second. Hence we recover $A(x,y)$ by deriving w.r.t. $\frac{\partial}{\partial z}$ and $\frac{\partial}{\partial \li{z}}$. Now let $A \in S^1(\Hc_{\nu})$ and $A \geq 0$. Then for any $z \in \D$
$$A(z,z) = A(K_z^{\nu})(z) = \langle A(K_z^{\nu}), K_z^{\nu} \rangle \geq 0,$$
and thus we get
\begin{flalign*}
& \int_{\D} |R_{\nu}(A)(z)| d \iota(z) = \int_{\D} |A(z,z) (1 - |z|^2)^{\nu}| d \iota(z)
\\ & = \frac{1}{\nu - 1} \int_{\D} A(z,z) d \iota_{\nu}(z) = \frac{1}{\nu-1} \Tr(A),
\end{flalign*}
so $R_{\nu}(A) \in L^1(\D)$. Now if $A \in S^1(\Hc_{\nu})$ then $A = A_1 + i A_2$ with $A_1$ and $A_2$ self-adjoint, and writing each self-adjoint operator as $A_j = A_j^+ - A_j^-$, where $A_j^{\pm}$ are positive, we obtain
\begin{flalign*}
& \int_{\D} |R_{\nu}(A)(z)| d \iota(z) \leq \frac{2}{\nu - 1} \nm{A}_{1}.
\end{flalign*}
This proves the fact for $S^1(\Hc_{\nu})$.

Now let $A \in B(\Hc_{\nu})$. First we note that for any $z \in \D$
\begin{flalign*}
& |A(z,z)| = | \langle A( K_z^{\nu}), K_z^{\nu} \rangle_{\nu} | \leq \nm{A(K_z^{\nu})}_{\nu} \nm{K_z^{\nu}}_{\nu} \leq \nm{A} \cdot \nm{K_z^{\nu}}_{\nu}^2,
\end{flalign*}
implying
$$ |A(z,z)(1 - |z|^2)^{\nu}| \leq \nm{A}.$$
This proves the claim for $B(\Hc_{\nu})$.

The claim for general $S^p(\Hc_{\nu})$ follows by interpolation, using Theorem 2.2.4, Theorem 2.2.6 and Theorem 2.2.7 in \cite{zhu}.
\end{proof}

In particular
$$R_{\nu}\colon S^2(\Hc_{\nu}) \rightarrow L^2(\D, d \iota)$$
is continuous. Note that $S^2(\Hc_{\nu}) = \Hc_{\nu} \otimes \li{ \Hc_{\nu} }$ by considering the kernels of the operators, and it is a Hilbert space. We prove a lemma.

\begin{lemma}
\label{Toeplitzrmu}
We have
$$R_{\nu}^*(f) = \frac{1}{\nu - 1} T_f,$$
where $T_f$ is the Toeplitz operator on $\Hc_{\nu}$, and $R_{\nu}^*$ has kernel
$$R_{\nu}^*(f)(x,y) = \int_{\D} \frac{K^{\nu}(x,z) f(z) K^{\nu}(z,y)}{K^{\nu}(z,z)} d \iota(z).$$
\end{lemma}

\begin{proof}
The Toeplitz operator is defined by $T_f = P_{\Hc_{\nu}} M_f P_{\Hc_{\nu}}$, where $M_f$ is multiplication by $f$ in $L^2(\D, d \iota)$ and $P_{\Hc_{\nu}}$ is projection onto the space $\Hc_{\nu}$. We note that the map
$$f \mapsto (x \mapsto \int_{\D} (\nu + 1) \frac{K^{\nu}(x,z) f(z)}{K^{\nu}(z,z)} d\iota(z))$$
is the projection onto $\Hc_{\nu}$, so the Toeplitz operator $T_f$ has kernel

$$T_f(x,y) = \int_{\D} K^{\nu}(x,z) f(z) K^{\nu}(z,y) d \iota_{\nu}(z).$$
Now for $A \in S^2(\Hc_{\nu})$ and $f \in L^2(\D, d \iota)$

\begin{flalign*}
& \int_{\D} A R_{\nu}^*(f)^* (z,z) d \iota_{\nu} (z) = \langle A, R_{\nu}^*(f) \rangle_{2} = \langle R_{\nu}(A), f \rangle_{L^2(\D, d \iota)}
\\ & = \int_{\D} A(z,z) (1 - |z|^2)^{\nu} \li{f(z)} d \iota(z) = \frac{1}{\nu - 1} \int_{\D} A(z,z) \li{f(z)} d \iota_{\nu}(z)
\\ & = \frac{1}{\nu - 1} \int_{\D^3} A(y,x) K^{\nu}(x,z) \li{f(z)} K^{\nu}(z,y) d \iota_{\nu}(z) d \iota_{\nu}(x) d \iota_{\nu} (y)
\\ & = \int_{\D^3} A(y,x) \frac{K^{\nu}(x,z) \li{f(z)} K^{\nu}(z,y)}{K^{\nu}(z,z)} d \iota(y) d \iota_{\nu}(x) d \iota_{\nu} (z),
\end{flalign*}
implying that indeed
$$R_{\nu}^*(f)(x,y) = \int_{\D} \frac{K^{\nu}(x,z) f(z) K^{\nu}(z,y)}{K^{\nu}(z,z)} d \iota(z).$$
\end{proof}

\begin{remark}
\label{Rnustargeneral}
We remark that $R_{\nu}^*: L^p(\D, d \iota) \rightarrow S^p(\Hc_{\nu})$ can be defined by the same formula for $p > 1$ by using the adjoint of $R_{\nu}$ in the sense of Banach spaces. The proof is the same.
\end{remark}

Now we define the Berezin transform.

\begin{definition}
\label{berezindef}
The Berezin transform is defined as
\begin{flalign*}
& B_{\nu} = R_{\nu} R_{\nu}^*\colon L^2(\D, d \iota) \rightarrow L^2(\D, d \iota).
\end{flalign*}
\end{definition}

We calculate this more explicitly. We see that

\begin{flalign*}
& B_{\nu}(f)(z) = R_{\nu} R_{\nu}^*(f)(z) = (1 - |z|^2)^{\nu} R_{\nu}^*(f)(z,z)
\\ & = \frac{(1 - |z|^2)^{\nu}}{\nu - 1} \int_{\D} K^{\nu}(z,x) f(x) K^{\nu}(x,y) d \iota_{\nu}(x)
\\ & = \int_{\D} \frac{(1 - |z|^2)^{\nu} (1 - |x|^2)^{\nu}}{ (1 - z \li{x})^{\nu} (1 - x \li{z})^{\nu}} f(x) d \iota(x).
\end{flalign*}
Note that

$$B_{\nu}(f)(0) = \int_{\D} (1 - |x|^2)^{\nu} f(x) d \iota(x).$$
The open disk $\D$ is a homogeneous space for $SU(1,1)$, so $SU(1,1)$ acts transitively, as in Remark \ref{repsu11facts}, and the maps $R_{\nu}$, $R_{\nu}^*$ and thus also $B_{\nu}$ are $SU(1,1)$-invariant, which means that $B_{\nu}(f)(g \cdot 0) = B_{\nu}(g^{-1} f)(0)$. Thus it is usually enough to prove a statement in the point $0$.

It follows immediately that $B_{\nu}(f)$ is continuous for $f \in C_c(\D)$. We also note that we can define $B_{\nu}\colon L^1(\D, d \iota) \rightarrow L^1(\D, d \iota)$, by evaluating

\begin{flalign*}
& \int_{\D} |B_{\nu}(f)(z)| d \iota(z)  \leq \int_{\D} \int_{\D} | \frac{K^{\nu}(z,x)K^{\nu}(x,z)}{K^{\nu}(z,z)K^{\nu}(x,x)} f(x)| d \iota(z) d \iota(x)
\\ & = (\nu + 1) \int_{\D} \langle K_x, K_x \rangle |f(x)| \frac{d \iota(x)}{K^{\nu}(x,x)} = (\nu + 1) \nm{f}_{1}.
\end{flalign*}

We can also define $B_{\nu}\colon L^{\infty}(\D, d \iota) \rightarrow L^{\infty}(\D, d \iota)$, as

\begin{flalign*}
& |B_{\nu}(f)(z)| = | \int_{\D} \frac{K^{\nu}(z,x)K^{\nu}(x,z)}{K^{\nu}(z,z)K^{\nu}(x,x)} f(x) d \iota(x)|
\\ & \leq \nm{f}_{\infty} \int_{\D} \frac{K^{\nu}(z,x)K^{\nu}(x,z)}{K^{\nu}(z,z)K^{\nu}(x,x)} d \iota(x) = (\nu + 1) \nm{f}_{\infty}
\end{flalign*}
This means that in fact, $B_{\nu}: L^p(\D, d \iota) \rightarrow L^p(\D, d \iota)$ is defined and continuous for any $1 \leq p \leq \infty$ by interpolation using Theorem 2.2.4, Theorem 2.2.6 and Remark 2.2.5 in \cite{zhu}. We gather some more facts about the Berezin transform.
\begin{lemma}
\label{berezinto1}
For $f \in C(\D) \cap L^1(\D, d \iota)$ the function $(\nu - 1) B_{\nu}(f)$ converges to $f$ pointwise, and if $f \in L^1(\D, d \iota)$ then for any $\nu \geq 2$

$$ \int_{\D} (\nu - 1) B_{\nu}(f)(z) d \iota(z) = \int_{\D} f(z) d \iota(z).$$
and
$$ \Tr(R_{\nu}^*(f)) = \int_{\D} f(z) d \iota(z).$$
\end{lemma}

\begin{proof}
If $f \in C(\D) \cap L^1(\D, d \iota)$ and $x \in \mathbb{D}$ such that $x = g \cdot 0$ for $g \in SU(1,1)$, then

\begin{flalign*}
& R_{\nu} R_{\nu}^*(f)(x) = R_{\nu} R_{\nu}^*(f)(g \cdot 0) = R_{\nu} R_{\nu}^*(g^{-1} f)(0),
\end{flalign*}
so it suffices to prove $(\nu - 1) B_{\nu}(f)(0)$ converges to $f(0)$. If we choose $\delta > 0$ such that $|f(z) - f(0)| < \epsilon$ for $|z| < \delta$ we get

\begin{flalign*}
& |(\nu - 1) B_{\nu}(f)(0) - f(0)|
\\ & = (\nu - 1) | \int_{\D} K^{\nu}(0,z) f(z) K^{\nu}(z,0)(1 - |z|^2)^{\nu} d \iota(z) - f(0)|
\\ & \leq (\nu - 1) \int_{B_{\delta}} |f(z) - f(0)| (1 - |z|^2)^{\nu} d \iota(z)
\\ & + (\nu - 1) \int_{\D \backslash B_{\delta}} |f(z) - f(0)| (1 - |z|^2)^{\nu} d \iota(z)
\\ & \leq \epsilon + \int_{\D \backslash B_{\delta}} (\nu - 1) |f(z) - f(0)| (1 - |z|^2)^{\nu} d \iota(z).
\end{flalign*}
For $z \in \D \backslash B_{\delta}$ we see

$$\lim_{\nu \rightarrow \infty} (\nu - 1) |f(z) - f(0)| (1 - |z|^2)^{\nu} = 0,$$
so by Lebesgue's dominated convergence theorem
$$\lim_{\nu \rightarrow \infty} \int_{\D \backslash B_{\delta}} (\nu - 1) |f(z) - f(0)| (1 - |z|^2)^{\nu} d \iota(z) = 0.$$
Hence
$$\lim_{\nu \rightarrow \infty} B_{\nu}(f)(0) = f(0).$$

Now we prove the second part. First we calculate
\begin{flalign*}
& \Tr(R_{\nu}^*(f))  = \int_{\D} (\nu - 1) B_{\nu}(f)(z) d \iota(z)
\\ & = \int_{\D^2} \frac{K^{\nu}(z,x)K^{\nu}(x,z)}{K^{\nu}(x,x)} f(x) d \iota_{\nu} (z) d \iota(x) = \int_{\D} f(x) d \iota(x).
\end{flalign*}
\end{proof}

We also need some theory on $L^2(\D, d \iota)$. Most of the following has been taken from the introductory chapter of \cite{helgGGA}.

\begin{definition}
For $b \in S^1 = \partial \D$, $\lambda \in \C$ and $z \in \D$ let
$$e_{\lambda, b}(z) = \left(\frac{1 - |z|^2}{|z-b|^2}\right)^{\frac{-i \lambda + 1}{2}}.$$
\end{definition}

Using this we define the following, which is some kind of Fourier transform. For functions $f$ on the open disk $\D$, $\lambda \in \C$ and $b \in \D$ we let
$$ \tilde{f}(\lambda, b) = \int_{D} f(z) e_{\lambda, b}(z) d z$$
whenever this is well-defined. Note that it is always well-defined when $f \in C_c(\D)$. We also recall that a holomorphic function $\psi$ is of exponential type $R$ if for each $N \in \mathbb{Z}^+$

$$\sup_{\lambda \in \C} e^{-R |\mathrm{Im}(\lambda)|} (1 + |\lambda|^2)^N |\psi(\lambda)| < \infty.$$

We now have the following Plancherel theorem and inversion theorem \cite[Introduction, Theorem 4.2]{helgGGA}).

\begin{proposition}
\label{helginv}
For functions $f \in C_c^{\infty}(\D)$ we have
$$f(z) = \frac{1}{2 \pi^2} \int_{\R} \int_{S^1} \tilde{f}(- \lambda, b) e_{\lambda,b}(z) |c(\lambda)|^{-2} db d \lambda.$$

Furthermore, the map $f \mapsto \tilde{f}$ is a bijection of $C_c^{\infty}(\D)$ onto the space of holomorphic functions $\psi(\lambda,b)$ of uniform exponential type satisfying
$$ \int_{S^1} e_{- \lambda,b}(z) \psi(\lambda, b) db = \int_{S^1} e_{\lambda,b}(z) \psi(-\lambda,b) db.$$

This map extends to an isometry of $L^2(\D, d \iota)$ onto $L^2(\R^+ \times S^1, |c(\lambda)|^{-2} db d \lambda )$, i.e. we have

$$ \int_{\D} |f(z)|^2 d \iota(z) = \frac{1}{2 \pi^2} \int_{\R} \int_{S^1} | \tilde{f}(\lambda,b)|^2 |c(\lambda)|^{-2} db d \lambda.$$
Here $c$ is the Harish-Chandra $c$-function, given by

$$c(\lambda) = \pi^{- \frac{1}{2}} \frac{ \Gamma( \frac{1}{2} i \lambda)}{\Gamma( \frac{1}{2} (i \lambda + 1))}.$$
Furthermore,

$$ |c(\lambda)|^{-2} = \frac{ \pi \lambda}{2} \tanh( \frac{\pi \lambda}{2}).$$
\end{proposition}

Now we study the $B_{\nu}$. From \cite[Lemma 3.37]{untup}, we get the following.

\begin{proposition}
\label{berezineigen}
We have
$$ (\nu - 1)B_{\nu}(e_{\lambda, b}) =  b_{\nu}(\lambda) e_{\lambda, b},$$
where
$$ b_{\nu}(\lambda) = \frac{ \Gamma(i \lambda + \nu - \frac{1}{2}) \Gamma(- i \lambda + \nu  - \frac{1}{2})}{\Gamma(\nu) \Gamma(\nu - 1)} = \frac{ |\Gamma(i \lambda + \nu  - \frac{1}{2})|^2}{\Gamma(\nu) \Gamma(\nu - 1)}.$$
\end{proposition}

We want to study $b_{\nu}(\lambda)$ further. We note that
$$|\Gamma(i \lambda + \nu  - \frac{1}{2})|^2 = \frac{\pi}{\cosh(\pi \lambda)} \prod_{k = 1}^{\nu - 1} ( (k - \frac{1}{2})^2 + \lambda^2).$$
We prove a lemma.

\begin{lemma}
\label{boundsbnulambda}
For any $\lambda \in \R$
$$ \lim_{\nu \rightarrow \infty} b_{\nu}(\lambda) = 1.$$
Furthermore, $|\lambda_1| \leq |\lambda_2|$ implies
$$ b_{\nu}(\lambda_2) \leq b_{\nu}(\lambda_1),$$
and thus $b_{\nu}(\lambda)$ is uniformly bounded in $\nu$ and $\lambda$.
\end{lemma}

\begin{proof}
Note that

\begin{flalign*}
& b_{\nu}(\lambda) = \frac{ \Gamma(i \lambda + \nu - \frac{1}{2}) \Gamma(- i \lambda + \nu  - \frac{1}{2})}{\Gamma(\nu) \Gamma(\nu - 1)} = \frac{\frac{\pi}{\cosh(\pi \lambda)} \prod_{k = 1}^{\nu - 1} ( (k - \frac{1}{2})^2 + \lambda^2)}{\Gamma(\nu) \Gamma(\nu - 1)}
\\ & = \frac{\frac{\pi}{\cosh(\pi \lambda)} \prod_{k = 1}^{\nu - 1} ( (k - \frac{1}{2})^2 + \lambda^2)}{\prod_{k=1}^{\nu - 1} (k - \frac{1}{2})^2} \frac{\prod_{k=1}^{\nu - 1} (k - \frac{1}{2})^2}{\Gamma(\nu) \Gamma(\nu - 1)}
\\ & = \frac{\pi}{\cosh(\pi \lambda)} \prod_{k = 1}^{\nu - 1} ( 1 + \frac{(\pi \lambda)^2}{(k - \frac{1}{2})^2 \pi^2}) \frac{\prod_{k=1}^{\nu - 1} (k - \frac{1}{2})^2}{\Gamma(\nu) \Gamma(\nu - 1)}.
\end{flalign*}
Now we note

$$ \cosh(x) = \prod_{k=1}^{\infty} (1 + \frac{x^2}{(k - \frac{1}{2})^2 \pi^2}).$$
This immediately implies that $|\lambda_1| \leq |\lambda_2|$ gives
$b_{\nu}(\lambda_2) \leq b_{\nu}(\lambda_1)$ and
$$\lim_{\nu \rightarrow \infty} \frac{\pi}{\cosh(\pi \lambda)} \prod_{k = 1}^{\nu - 1} ( 1 + \frac{(\pi \lambda)^2}{(k - \frac{1}{2})^2 \pi^2}) = \pi.$$ For the term $\frac{\prod_{k=1}^{\nu - 1} (k - \frac{1}{2})^2}{\Gamma(\nu) \Gamma(\nu - 1)}$ we see

\begin{flalign*}
& \frac{\prod_{k=1}^{\nu - 1} (k - \frac{1}{2})^2}{\Gamma(\nu) \Gamma(\nu - 1)} = \frac{ \Gamma(\nu - \frac{1}{2})^2}{\pi \Gamma(\nu) \Gamma(\nu - 1)}.
\end{flalign*}
Then we note that $\Gamma(x + \alpha) \sim \Gamma(x) x^{\alpha}$ as $x$ is going to infinity. Hence
\begin{flalign*}
& \lim_{\nu \rightarrow \infty} \frac{\prod_{k=1}^{\nu - 1} (k - \frac{1}{2})^2}{\Gamma(\nu) \Gamma(\nu - 1)} = \lim_{\nu \rightarrow \infty} \frac{ \Gamma(\nu - \frac{1}{2})^2}{\pi \Gamma(\nu) \Gamma(\nu - 1)}
\\ & = \lim_{\nu \rightarrow \infty} \frac{\Gamma(\nu - 1) (\nu^{\frac{1}{2}})^{-2}}{\pi \Gamma(\nu - 1)^2 \nu^{-1}} = \frac{1}{\pi}.
\end{flalign*}
We conclude that
$$ \lim_{\nu \rightarrow \infty} b_{\nu}(\lambda) = 1.$$
\end{proof}

We see that for any $\lambda \in \R$
$$ b_{\nu}(\lambda) \leq b_{\nu}(0) = \frac{\Gamma(\nu - \frac{1}{2})^2}{\Gamma(\nu) \Gamma(\nu - 1)} \leq 1,$$
where the last inequality is by the log-convexity of the Gamma function. Thus also the norm of $(\nu - 1)B_{\nu}$ is bounded by $1$. We make this a corollary; a similar result with this constant was obtained in \cite[Proposition 3.1]{zhangCon}.

\begin{corollary}
The map
$$B_{\nu}\colon L^2(\D, d \iota) \rightarrow L^2(\D, d \iota)$$
has the property $\nm{B_{\nu}} \leq b_{\nu}(0)^2 = (\frac{\Gamma(\nu - \frac{1}{2})^2}{\Gamma(\nu) \Gamma(\nu - 1)})^2 \leq 1$.
\end{corollary}

\begin{proof}
First we remark that for any $f \in L^2(\D)$
$$ \widetilde{B_{\nu}(f)}(\lambda,b) =  b_{\nu}(\lambda) \tilde{f}(\lambda,b).$$
Applying the Plancherel theorem, Proposition \ref{helginv}, we get
\begin{flalign*}
& \nm{B_{\nu}(f)}^2 = \frac{1}{2 \pi^2} \int_{\R} \int_{S^1} |b_{\nu}(\lambda) \tilde{f}(\lambda,b)|^2 |c(\lambda)|^{-2} db d \lambda
\\ & \leq b_{\nu}(0)^2 \frac{1}{2 \pi^2} \int_{\R} \int_{S^1} |\tilde{f}(\lambda,b)|^2 |c(\lambda)|^{-2} db d \lambda = b_{\nu}(0)^2 \nm{f}^2.
\end{flalign*}
\end{proof}

We prove a bound for $(\nu - 1) R_{\nu}$.

\begin{proposition}
\label{Rmubdd}
For any $\nu > 1$ we have
$$ \frac{1}{\nu - 1} \nm{(\nu - 1) R_{\nu}^*(f)}_{2}^2 \leq \nm{f}_{2}^2.$$
\end{proposition}

\begin{proof}
We see that
\begin{flalign*}
& \frac{1}{\nu - 1} \nm{(\nu - 1) R_{\nu}^*(f)}_{2}^2 = \langle (\nu - 1) R_{\nu}^*(f), R_{\nu}^*(f) \rangle_{2}
\\ & = \langle (\nu - 1) B_{\nu - 1}(f), f \rangle \leq \nm{f}_{2}^2.
\end{flalign*}
This gives the result.
\end{proof}

\section{Quantum channels}
\label{channelingsec}

In this section we define the quantum channels we will study. Usually quantum channels are defined for finite-dimensional spaces \cite{petz}, and are then defined to be completely positive trace-preserving maps (CPTP). In the infinite-dimensional case, less has been done. In \cite{haap} channels are defined as completely positive unital weak-$*$ continuous maps between $C^*$-algebras. The maps $\mathcal{T}_{\mu,k}^{\nu}$ we define in Definition \ref{channeldef} are clearly unital, and we prove they are completely positive and trace-preserving up to a constant in Proposition \ref{Tquantumchannel}, like the channels in \cite{haas}. We prove they are weak-$*$ continuous in Proposition \ref{qtchannel}. In fact, for a compact group equivariant quantum channels on irreducible representations which are trace-preserving have to be unital up to a constant, and vice versa. Hence we call the $\mathcal{T}_{\mu,k}^{\nu}$ quantum channels.

\begin{definition}
\label{channeldef}
Associated to the Decomposition (\ref{TensorProdDec}) we define
$$\mathcal{T}_{\mu,k}^{\nu}\colon B(\Hc_{\mu}) \rightarrow B(\Hc_{\mu + \nu + 2 k})$$
by $\mathcal{T}_{\mu,k}^{\nu}(A) = P_k (A \otimes I_{\nu}) P_k^*$. Here $I_{\nu}$ is the identity operator on $\Hc_{\nu}$.
\end{definition}

We claim this map sends trace-class operators to trace-class operators and is trace-preserving up to a constant. We first need some theory, and normalize the Haar measure on $SU(1,1)$ such that
$$ \int_{SU(1,1)} f(g) dg = \int_{\D} \int_{U(1)} f(zk) dk d \iota(z),$$
where we remember that $\D = SU(1,1)/U(1)$. Here $dk$ is normalized to integrate to $1$. Now we need the following version of Schur's lemma. It is true for general discrete series representations \cite{hari} and our result could be obtained by general theory, but for completeness we give an elementary proof using our normalization of the Haar measure.

\begin{proposition}
The representations $\Hc_{\nu}$ have square integrable matrix coefficients and for $v_1, v_2, w_1, w_2 \in \Hc_{\nu}$
$$ \int_{SU(1,1)} \langle g \cdot v_1, w_1 \rangle_{\nu} \li{ \langle g \cdot v_2, w_2 \rangle}_{\nu} dg = \frac{1}{\nu-1} \langle v_1, v_2 \rangle_{\nu} \li{ \langle w_1, w_2 \rangle}_{\nu}.$$
\end{proposition}

\begin{proof}
By \cite[Proposition 9.6]{knapp} there is some constant $C$ such that
$$ \int_{SU(1,1)} \langle g \cdot v_1, w_1 \rangle_{\nu} \li{ \langle g \cdot v_2, w_2 \rangle}_{\nu} dg = C \langle v_1, v_2 \rangle_{\nu} \li{ \langle w_1, w_2 \rangle}_{\nu}.$$
We calculate $C$ by evaluating at $v_1 = v_2 = w_1 = w_2 = 1$. We see that $g \cdot 1 (z) = (-b z + \li{a})^{- \nu}$ for $g = \begin{pmatrix} a & b \\ \li{b} & \li{a} \end{pmatrix}$. Hence
\begin{flalign*}
& \langle g \cdot 1, 1 \rangle_{\nu} = \int_{\D} (- b z + \li{a})^{-\nu} d \iota_{\nu}(z) = \li{a}^{- \nu} \langle K_{ \frac{\li{b}}{a}}, 1 \rangle_{\nu} = \li{a}^{-\nu},
\end{flalign*}
and
\begin{flalign*}
& \int_{SU(1,1)} \langle g \cdot 1, 1 \rangle_{\nu} \li{ \langle g \cdot 1,1 \rangle_{\nu}} dg = \int_{SU(1,1)} |a|^{-2 \nu} d g.
\end{flalign*}
This is right-invariant under $K = U(1) = \{ \begin{pmatrix} e^{i \theta} & 0 \\ 0 & e^{-i \theta} \end{pmatrix} \mid \theta \in \R \}$, and we get an integral over $\D \cong SU(1,1) / U(1)$. Here $g$ is identified with
$$z \coloneqq g \cdot 0 = - \frac{\li{b}}{\li{a}} \in \D,$$
and $|z|^2 = \frac{|b|^2}{|a|^2} = \frac{|a|^2 - 1}{|a|^2}$, implying that $|a|^2 = \frac{1}{1 - |z|^2}$. Hence
$$ C = \int_{SU(1,1)} |a|^{-2 \nu} d g = \int_{\D} (1 - |z|^2)^{\nu} d \iota(z) = \frac{1}{\nu - 1}.$$
\end{proof}

We want to say something about the trace of $\mathcal{T}_{\mu,k}^{\nu}(A)$. First we need a lemma, restating Schur's lemma for $L^2$-matrix coefficients and trace-class operators.

\begin{lemma}
\label{traceclassint}
A positive operator $A \in B(\Hc_{\nu})$ is trace-class if and only if $\int_{SU(1,1)} \langle g A g^{-1} v, v \rangle_{\nu} dg < \infty$ for all $v \in \Hc_{\nu}$. In that case
$$\int_{SU(1,1)} gAg^{-1} dg = \frac{\Tr(A)}{\nu - 1} I_{\nu}.$$
\end{lemma}

\begin{proof}
First we prove the left to right implication. We assume $A = x \otimes x^*$, then we see that the integral $\int_{SU(1,1)} gAg^{-1} dg$ converges and
\begin{flalign*}
& \int_{SU(1,1)} \langle g A g^{-1} v, w \rangle_{\nu} dg = \int_{SU(1,1)} \langle g \cdot x, w \rangle_{\nu} \li{ \langle g \cdot x, v \rangle_{\nu} } dg
\\ & = \frac{1}{\nu - 1} \nm{x}_{\nu}^2 \li{\langle w, v \rangle_{\nu}},
\end{flalign*}
implying that
$$\int_{SU(1,1)} gAg^{-1} dg = \frac{\langle x, x \rangle_{\nu}}{\nu - 1} I_{\nu} = \frac{\Tr(A)}{\nu - 1} I_{\nu}.$$
This gives $\int_{SU(1,1)} \langle g A g^{-1} v, v \rangle_{\nu} dg = \frac{\Tr(A)}{\nu - 1} \langle v, v \rangle_{\nu} < \infty$. The case for $A$ a positive trace-class operator follows by the spectral decomposition.

Now assume $A$ is positive but not trace-class. Let $(e_i)_{i=1}^{\infty}$ be an orthonormal basis for $\Hc_{\nu}$ and
$$A_n = \sum_{i=1}^n \langle A e_i, e_i \rangle_{\nu} e_i \otimes e_i^*.$$
Then we see that $A \geq A_n$ and thus also $gAg^{-1} \geq gA_n g^{-1}$ for all $g \in SU(1,1)$. Note that $A_n$ is trace-class. Then
\begin{flalign*}
& \int_{SU(1,1)} \langle gAg^{-1}v, v \rangle_{\nu} d g \geq \int_{SU(1,1)} \langle gA_n g^{-1} v, v \rangle_{\nu} d g
\\ & = \frac{\Tr(A_n)}{\nu - 1} \langle v, v \rangle_{\nu} = \frac{\langle v, v \rangle}{\nu - 1} \sum_{i=1}^n \langle A e_i, e_i \rangle_{\nu},
\end{flalign*}
so the integral $\int_{SU(1,1)} \langle gAg^{-1} v, v \rangle_{\nu} dg$ diverges. We have proven our statement.
\end{proof}

We prove that $\mathcal{T}_{\mu,k}^{\nu}$ is trace preserving up to a constant.

\begin{proposition}
\label{Tquantumchannel}
The map $\mathcal{T}_{\mu,k}^{\nu}$ is completely positive and $A \in S^1(\Hc_{\mu})$ implies $\mathcal{T}_{\mu,k}^{\nu}(A) \in S^1(\Hc_{\mu + \nu + 2k})$. In that case
$$\Tr(\mathcal{T}_{\mu,k}^{\nu}(A)) = \frac{\mu + \nu + 2k -1}{\mu - 1} \Tr(A).$$
\end{proposition}

\begin{proof}
It is obvious that the map $\mathcal{T}_{\mu,k}^{\nu}$ is completely positive, as it is the composition of the $*$-homomorphism $A \mapsto A \otimes I$ and $A \mapsto P_k A  P_k^*$. Now let $A \in S^1(\Hc_{\mu})$ be a trace-class operator. Assume $A$ is positive, then so is $\mathcal{T}_{\mu,k}^{\nu}(A)$. By Lemma \ref{traceclassint} we see that
$$\int_{SU(1,1)} g A g^{-1} d g = \frac{\Tr(A)}{\mu - 1}.$$
Thus
\begin{flalign*}
& \int_{SU(1,1)} g \mathcal{T}_{\mu,k}^{\nu}(A) g^{-1} dg = \int_{SU(1,1)} g P_k(A \otimes I_{\nu})P_k^* g^{-1} dg
\\ & = P_k \int_{SU(1,1)} g(A \otimes I_{\nu})g^{-1} dg P_k^* = \frac{\Tr(A)}{\mu - 1} I_{\mu + \nu + 2k}.
\end{flalign*}
We can now deduce by Lemma \ref{traceclassint} that $\mathcal{T}_{\mu,k}^{\nu}(A)$ is trace-class, and that its trace must be
$$ \Tr(\mathcal{T}_{\mu,k}^{\nu}(A)) = \frac{\mu + \nu + 2k - 1}{\mu-1} \Tr(A).$$
For general $A$, we write $A$ as a linear combination of positive operators, which will again be trace-class, and use linearity of $\mathcal{T}_{\mu,k}^{\nu}$ and the trace.
\end{proof}

We prove some more properties of this operator.

\begin{proposition}
\label{qtchannel}
The map $\mathcal{T}_{\mu,k}^{\nu}$ is unital, weak-$*$ continuous, and it sends compact operators to compact operators. As an operator
$$ \mathcal{T}_{\mu,k}^{\nu}\colon B(\Hc_{\mu}) \rightarrow B(\Hc_{\mu + \nu + 2k}),$$
we have $\nm{\mathcal{T}_{\mu,k}^{\nu}} = 1$. Moreover, as an operator
$$ \mathcal{T}_{\mu,k}^{\nu}\colon S^p(\Hc_{\mu}) \rightarrow S^p(\Hc_{\mu + \nu + 2k}) $$
we have $\nm{\mathcal{T}_{\mu,k}^{\nu}}_{p \rightarrow p} \leq (2 \frac{\mu + \nu + 2k - 1}{\mu - 1})^{\frac{1}{p}}.$
\end{proposition}

\begin{proof}
$\mathcal{T}_{\mu,k}^{\nu}$ is obviously unital. We see it is weak-$*$ continuous realizing $\rho \in S^1(\Hc_{\mu + \nu + 2k})$ implies
$$ \Tr(\mathcal{T}_{\mu,k}^{\nu}(A) \rho) = \Tr(A \otimes I_{\nu} (P_k^* \rho P_k))$$
together with the fact that
$$A \mapsto A \otimes I_{\nu}$$
is weak-$*$ continuous. We also see
$$\nm{\mathcal{T}_{\mu,k}^{\nu}(A)} = \nm{P_k (A \otimes I) P_k^*} \leq \nm{P_k} \cdot \nm{A \otimes I} \cdot \nm{P_k^*} \leq \nm{A}.$$
$\nm{\mathcal{T}_{\mu,k}^{\nu}} = 1$ follows by letting $A = I_{\mu}$.

Now let $A \in B(\Hc_{\mu + \nu + 2k})$ be a compact operator and $\{A_n \}_n$ a sequence of finite rank operators converging to $A$. Then each $A_n$ is necessarily trace-class and by Proposition \ref{Tquantumchannel} we see that $\mathcal{T}_{\mu,k}^{\nu}(A_n)$ is trace-class. This implies $\mathcal{T}_{\mu,k}^{\nu}(A_n)$ is compact and that $\mathcal{T}_{\mu,k}^{\nu}(A)$ is compact as the limit of a sequence of compact operators.

We now prove a bound on $\nm{\mathcal{T}_{\mu,k}^{\nu}}_{1 \rightarrow 1}$. Writing $A \in S^1(\Hc_{\mu})$ as $A = A_1 + i A_2$, where both $A_1$ and $A_2$ are self-adjoint, we get

$$ \nm{A}_{1} \leq \nm{A_1}_{1} + \nm{A_2}_{1},$$
and
$$ \nm{A_1}_{1} = \nm{\frac{A + A^*}{2}}_{1} \leq \nm{A}_{1},$$
and similarly for $A_2$. Hence we see, using the decomposition of $A_1$ and $A_2$ in positive and negative parts
\begin{flalign*}
& \nm{\mathcal{T}_{\mu,k}^{\nu}(A)}_{1} \leq \nm{\mathcal{T}_{\mu,k}^{\nu}(A_1)}_{1} + \nm{\mathcal{T}_{\mu,k}^{\nu}(A_2)}_{1}
\\ & = \frac{\mu + \nu + 2k -1}{\mu - 1} ( \nm{A_1}_{1} + \nm{A_2}_{1})
\\ & \leq 2 \frac{\mu + \nu + 2k -1}{\mu - 1} \nm{A}_{1}.
\end{flalign*}
The bound for $\nm{\mathcal{T}_{\mu,k}^{\nu}}_{p \rightarrow p}$ follows by interpolation using Theorem 2.2.4, Remark 2.2.5 and Theorem 2.2.7 in \cite{zhu}.
\end{proof}

We now want to calculate the kernel of $\mathcal{T}_{\mu,k}^{\nu}(A)$, as we eventually want to find the functional calculus of $\mathcal{T}_{\mu,k}^{\nu}(A)$ for any $A \in S^2(\Hc_{\mu})$.

\begin{lemma}
\label{gammaform}
Let $A(z,w)$ be the integral kernel of an operator $A \in B(\Hc_{\mu})$. Then
\begin{flalign*}
& \mathcal{T}_{\mu,k}^{\nu}(A)(x,y) = C_{\mu, \nu, k}^2 \int_{\D^3} A(z,u) (1 - u \li{y})^{- \mu - k} (1 - w \li{y})^{- \nu - k}
\\ & (u - w)^k (\li{z} - \li{w})^k (1 - x \li{z})^{-\mu - k} (1 - x \li{w})^{-\nu - k} d \iota_{\mu}(z) d \iota_{\nu}(w) d \iota_{\mu}(u).
\end{flalign*}
\end{lemma}

\begin{proof}
This is a direct calculation.

\begin{flalign*}
& \mathcal{T}_{\mu,k}^{\nu}(A)(x,y) = \mathcal{T}_{\mu,k}^{\nu}(A)(K_y)(x) = P_k (A \otimes I) P_k^*(K_y)(x)
\\ & = C_{\mu, \nu, k} \int_{\D^2} (A \otimes I) P_k^*(K_y) (z,w) (\li{z} - \li{w})^k (1 - x \li{z})^{-\mu - k} 
\\ & (1 - x \li{w})^{-\nu - k} d \iota_{\mu}(z) d \iota_{\nu}(w)
\\ & = C_{\mu, \nu, k} \int_{\D^3} A(z,u) P_k^*(K_y) (u,w) (\li{z} - \li{w})^k (1 - x \li{z})^{-\mu - k}
\\ & (1 - x \li{w})^{-\nu - k} d \iota_{\mu}(z) d \iota_{\nu}(w) d \iota_{\mu}(u)
\\ & = C_{\mu, \nu, k}^2 \int_{\D^4} A(z,u) K_y(\zeta) (1 - u \li{\zeta})^{- \mu - k} (1 - w \li{\zeta})^{- \nu - k} (u - w)^k 
\\ & (\li{z} - \li{w})^k (1 - x \li{z})^{-\mu - k} (1 - x \li{w})^{-\nu - k} d \iota_{\mu + \nu + 2k}(\zeta) d \iota_{\mu}(z) d \iota_{\nu}(w) d \iota_{\mu}(u)
\\ & = C_{\mu, \nu, k}^2 \int_{\D^3} A(z,u) (1 - u \li{y})^{- \mu - k} (1 - w \li{y})^{- \nu - k} (u - w)^k 
\\ & (\li{z} - \li{w})^k (1 - x \li{z})^{-\mu - k} (1 - x \li{w})^{-\nu - k} d \iota_{\mu}(z) d \iota_{\nu}(w) d \iota_{\mu}(u).
\end{flalign*}
\end{proof}

We compute the covariant symbol of $\mathcal{T}_{\mu,k}^{\nu}$ on the Toeplitz operator $R_{\mu}^*(f)$.

\begin{proposition}
\label{channelssumberezin}
We have
$$ R_{\mu + \nu + 2k} \mathcal{T}_{\mu,k}^{\nu} R_{\mu}^* (f) = C_{\mu, \nu, k}^2 \sum_{j = 0}^k (-1)^j \binom{k}{j} \frac{(\nu + k - j)_k}{(\nu)_k} B_{\mu + j}(f),$$
where $B_{\mu+j}$ is the Berezin transform of Definition \ref{berezindef}.
\end{proposition}

\begin{proof}
Note that by $SU(1,1)$-invariance we only need to prove that

\begin{flalign*}
& R_{\mu + \nu + 2k} \mathcal{T}_{\mu,k}^{\nu} R_{\mu}^* (f) (0)
\\ & = C_{\mu, \nu, k}^2 \sum_{j = 0}^k (-1)^j \binom{k}{j} \frac{(\nu + k - j)_k}{(\nu)_k} B_{\mu + j}(f)(0).
\end{flalign*}

First we calculate $\mathcal{T}_{\mu,k}^{\nu}(R_{\mu}(f))(x,y)$. By Lemma \ref{gammaform}

\begin{flalign*}
& \mathcal{T}_{\mu,k}^{\nu}(R_{\mu}^*(f))(x,y) = C_{\mu, \nu, k}^2 \int_{\D^3} R_{\mu}^*(f)(z,u) (1 - u \li{y})^{- \mu - k}
\\ & (1 - w \li{y})^{- \nu - k} (u - w)^k (\li{z} - \li{w})^k (1 - x \li{z})^{-\mu - k} (1 - x \li{w})^{-\nu - k} d \iota_{\mu}(z) d \iota_{\nu}(w) d \iota_{\mu}(u)
\\ & = \frac{ C_{\mu, \nu, k}^2}{\mu - 1} \int_{\D^4} K^{\mu}(z,v) f(v) K^{\mu}(v,u) (1 - u \li{y})^{- \mu - k} (1 - w \li{y})^{- \nu - k}
\\ & (u - w)^k (\li{z} - \li{w})^k (1 - x \li{z})^{-\mu - k} (1 - x \li{w})^{-\nu - k} d \iota_{\mu}(z) d \iota_{\nu}(w) d \iota_{\mu}(u) d \iota_{\mu}(v)
\\ & = \frac{C_{\mu, \nu, k}^2}{\mu - 1} \int_{\D^2} f(v) (1 - v \li{y})^{- \mu - k} (1 - w \li{y})^{- \nu - k} (v - w)^k 
\\ & (\li{v} - \li{w})^k (1 - x \li{v})^{-\mu - k} (1 - x \li{w})^{-\nu - k} d \iota_{\nu}(w) d \iota_{\mu}(v).
\end{flalign*}
It follows that
\begin{flalign*}
& R_{\mu + \nu + 2k} \mathcal{T}_{\mu,k}^{\nu} R_{\mu}^* (f) (0) = \frac{C_{\mu, \nu, k}^2}{\mu - 1} \int_{\D^2} f(v) |v - w|^{2k} d \iota_{\nu}(w) d \iota_{\mu}(v)
\\ & = \frac{C_{\mu, \nu, k}^2}{\mu - 1} \sum_{i=0}^k \binom{k}{i}^2 \frac{i!}{(\nu)_i} \int_{\D} f(v) |v|^{2(k-i)} d \iota_{\mu}(v)
\\ & = C_{\mu, \nu, k}^2 \sum_{i=0}^k \binom{k}{i}^2 \frac{i!}{(\nu)_i} \int_{\D} f(v) |v|^{2 (k-i)}(1 - |v|^2)^{\mu} d \iota(v).
\end{flalign*}
We expand this as

\begin{flalign*}
& R_{\mu + \nu + 2k} \mathcal{T}_{\mu,k}^{\nu} R_{\mu}^* (f) (0)
\\ & = C_{\mu, \nu, k}^2 \sum_{i=0}^k \binom{k}{i}^2 \frac{i!}{(\nu)_i} \int_{\D} f(v) |v|^{2(k-i)} (1 - |v|^2)^{\mu} d \iota(v)
\\ & = C_{\mu, \nu, k}^2 \sum_{i=0}^k \binom{k}{i}^2 \frac{i!}{(\nu)_i} \int_{\D} f(v) (1 - (1 - |v|^2))^{k-i} (1 - |v|^2)^{\mu} d \iota(v)
\\ & = C_{\mu, \nu, k}^2 \sum_{i=0}^k \sum_{j = 0}^{k-i} (-1)^j \binom{k}{i}^2 \frac{i!}{(\nu)_i} \binom{k-i}{j} \int_{\D} f(v) (1 - |v|^2)^{\mu + j} d \iota(v)
\\ & = C_{\mu, \nu, k}^2 \sum_{i=0}^k \sum_{j = 0}^{k-i} (-1)^j \binom{k}{i}^2 \frac{i!}{(\nu)_i} \binom{k-i}{j} B_{\mu + j}(f)(0)
\\ & = C_{\mu, \nu, k}^2 \sum_{j = 0}^k (-1)^j ( \sum_{i=0}^{k-j} \binom{k}{i}^2 \frac{i!}{(\nu)_i} \binom{k-i}{j} ) B_{\mu + j}(f)(0).
\end{flalign*}
Using Gauss's summation theorem we evaluate the coefficient of $B_{\mu+j}(f)(0)$ as

\begin{flalign*}
& \sum_{i=0}^{k-j} \binom{k}{i}^2 \frac{i!}{(\nu)_i} \binom{k-i}{j} = \frac{(-1)^j}{j!} \sum_{i=j}^{k-j} \frac{((- k)_i)^2 (- k + i)_j}{(\nu)_i i!}
\\ & = \frac{(-1)^j}{j!} \sum_{i=0}^{k-j} \frac{( -k)_{i + j} (-k)_i}{(\nu)_i i!} = \frac{(-k)_j (-1)^j}{j!} \sum_{i=0}^{k-j} \frac{ (-k)_i (-k + j)_i}{(\nu)_{i} i!}
\\ & = \binom{k}{j} {}_2F_1(-k,-k + j; \nu, 1) = \binom{k}{j} \frac{(\nu + k - j)_k}{(\nu)_k}.
\end{flalign*}
Hence

\begin{flalign*}
& R_{\mu + \nu + 2k} \mathcal{T}_{\mu,k}^{\nu} R_{\mu}^* (f) (0)
\\ & = C_{\mu, \nu, k}^2 \sum_{j = 0}^k (-1)^i ( \sum_{i=0}^{k-j} \binom{k}{i}^2 \frac{i!}{(\nu)_i} \binom{k-i}{j} ) B_{\mu + j}(f)(0)
\\ & = C_{\mu, \nu, k}^2 \sum_{j = 0}^k (-1)^j \binom{k}{j} \frac{(\nu + k - j)_k}{(\nu)_k} B_{\mu + j}(f)(0).
\end{flalign*}
This completes the proof.
\end{proof}

We collect some information.

\begin{proposition}
\label{Elimit}
We have
$$ R_{\mu + \nu + 2k} \mathcal{T}_{\mu,k}^{\nu} R_{\mu}^* = \frac{(\mu)_k (\nu)_k}{k! (\mu + \nu + k + 1)_k} \sum_{j = 0}^k (-1)^j \binom{k}{j} \frac{(\nu + k - j)_k}{(\nu)_k} B_{\mu + j}$$
and
$$ \lim_{\nu \rightarrow \infty} \nm{R_{\mu + \nu + 2k} \mathcal{T}_{\mu,k}^{\nu} R_{\mu}^* - \frac{ (\mu)_{k}}{k!} \sum_{j = 0}^k (-1)^j \binom{k}{j} B_{\mu + j}}_p = 0$$
for any $1 \leq p \leq \infty$.
\end{proposition}

\begin{proof}
The first line follows directly from Proposition \ref{channelssumberezin} and Proposition \ref{constantcmunuk} saying

$$C^{-2}_{\mu, \nu, k} =  \frac{k! (\mu + \nu + k + 1)_k}{(\mu)_k (\nu)_k}.$$
This implies that
$$ \lim_{\nu \rightarrow \infty} C^{-2}_{\mu, \nu, k} = \frac{k!}{(\mu)_k}.$$
We also observe that
$$ \lim_{\nu \rightarrow \infty} (-1)^j \binom{k}{j} \frac{(\nu + k - j)_k}{(\nu)_k} = (-1)^j \binom{k}{j}.$$
The second part follows.
\end{proof}

We define the following notation for convenience.

\begin{definition}
\label{noncptdefEmunuk}
$$E_{\mu,k}^{\nu} \coloneqq R_{\mu + \nu + 2k} \mathcal{T}_{\mu,k}^{\nu} R_{\mu}^*,$$
and
$$ E_{\mu,k} \coloneqq \frac{(\mu)_{k}}{k!} \sum_{j=0}^k (-1)^j \binom{k}{j} B_{\mu + j}.$$
\end{definition}

Note that by Proposition \ref{Elimit}
$$\lim_{\nu \rightarrow \infty} \nm{E_{\mu,k}^{\nu} - E_{\mu,k}}_{p} = 0$$
for any $1 \leq p \leq \infty$. We discuss another way of expressing our limit, which was done in \cite{soletal} for the group $SU(2)$ and its associated $SU(2)$-equivariant quantum channels. In that paper Aschieri, Ruba and Solovej defined a generalized Husimi function. We define such a function for $SU(1,1)$.

\begin{definition}
\label{genHus}
For $A \in B(\Hc_{\nu})$ we can define a generalized Husimi function
$H_{\nu}^i(A) \colon \D \rightarrow \C$
by
$$H_{\nu}^i(A)(g \cdot 0) \coloneqq \frac{(\nu)_i}{i!} \langle A g \cdot z^i, g \cdot z^i \rangle_{\nu}.$$
\end{definition}

Note that his is well-defined and $SU(1,1)$-invariant, i.e.
$$ g \cdot H_{\nu}^i(A) = H_{\nu}^i(gAg^{-1})$$
for $g \in SU(1,1)$. Note also that the $(\frac{(\nu)_i}{i!})^{\frac{1}{2}} z^i$ are eigenvectors for the Lie algebra element $H = \begin{pmatrix} 1 & 0 \\ 0 & -1 \end{pmatrix} \in \mathfrak{su}(1,1)$ with eigenvalues $\nu + 2i$. We prove a lemma.

\begin{lemma}
\label{Huscont}
The maps
$$H_{\nu}^i: B(\Hc_{\nu}) \rightarrow L^{\infty}(\D, d \iota)$$
and
$$H_{\nu}^i: S^p(\Hc_{\nu}) \rightarrow L^p(\D, d \iota)$$
are well-defined and continuous.
\end{lemma}

\begin{proof}
Note that for $A \in B(\Hc_{\nu})$ and $g \in SU(1,1)$
$$|H_{\nu}^i(A)(g \cdot 0)| \leq \nm{A},$$
which gives the first part. We now concentrate on the case $p=1$. Let $A \in S^1(\Hc_{\nu})$ and $A \geq 0$. Then $H_{\nu}^i(A)(g \cdot 0) \geq 0$ for every $g$ and
$$ \int_{\D} H_{\nu}^i(A)(z) d \iota(z) = \int_{SU(1,1)} H_{\nu}^i(A)(g \cdot 0) d \iota(z) = \Tr(A)$$
by Lemma \ref{traceclassint}, so $H_{\nu}^i(A) \in L^1(\D, d \iota)$. The claim then follows by writing each $A \in S^1(\Hc_{\nu})$ as a sum of positive operators. The claim for general $p$ follows by interpolation using Theorem 2.2.4, Remark 2.2.5, Theorem 2.2.6 and Theorem 2.2.7 in \cite{zhu}.
\end{proof}

We discuss another way of expressing $E_{\mu,k}(f)$.

\begin{proposition}
\label{Husimiremark1}
For any $f \in L^1(\D, d \iota)$
$$ E_{\mu,k}(f) = H_{\mu}^k(R_{\mu}^*(f)).$$
\end{proposition}

\begin{proof}
Note that in the proof of Proposition \ref{channelssumberezin}
$$ E_{\mu,k}^{\nu}(f)(0) = C_{\mu,\nu,k}^2 \sum_{i=0}^k \binom{k}{i}^2 \frac{i!}{(\nu)_i} \int_{\D} f(v) |v|^{2 (k-i)} (1 - |v|^2)^{\mu} d \iota(v).$$
Using Lemma \ref{Toeplitzrmu}
\begin{flalign*}
& H_{\mu}^i(R_{\mu}^*(f))( 0) = \frac{(\mu)_i}{i!} \langle R_{\mu}^*(f) z^i, z^i \rangle_{\mu}
\\ & = \frac{(\mu)_i}{i!} \int_{\D^2} K^{\mu}(x,y) f(y) K^{\mu}(y,z) (1 - |y|^2)^{\mu} \li{x}^i z^i d \iota_{\mu}(x) d \iota_{\mu}(z) d \iota(y)
\\ & = \frac{(\mu)_i}{i!} \int_{\D^2} f(y) |y|^{2 i} (1 - |y|^2)^{\mu} d \iota_{\mu}(x) d \iota(y).
\end{flalign*}
Hence
\begin{flalign*}
& E_{\mu,k}^{\nu}(f)(0) = C_{\mu,\nu,k}^2 \sum_{i=0}^k \binom{k}{i}^2 \frac{i! (k-i)!}{(\nu)_i (\mu)_{k-i}} H_{\mu}^{k-i}(R_{\mu}^*(f))(0),
\end{flalign*}
and thus
$$ E_{\mu,k}(f)(0) = \lim_{\nu \rightarrow \infty} E_{\mu,k}^{\nu} = \frac{(\mu)_k}{k!} \frac{k!}{(\mu)_{k}} H_{\mu}^k(R_{\mu}^*(f))(0) = H_{\mu}^k(R_{\mu}^*(f))(0),$$
where the limit can be taken in any $L^p$-norm. We conclude that
$$ E_{\mu,k}(f) = H_{\mu}^k(R_{\mu}^*(f)).$$
by $SU(1,1)$-invariance.
\end{proof}


Note that we want to describe the functional calculus of
\begin{flalign*}
& \mathcal{T}_{\mu,k}^{\nu}(R_{\mu}^*(f)) = R_{\mu + \nu + 2k}^{-1} E_{\mu ,k}^{\nu}(f) = R_{\mu + \nu + 2k}^* ( R_{\mu + \nu + 2k}^*)^{-1} R_{\mu + \nu + 2k}^{-1} E_{\mu,k}^{\nu}(f)
\\ & = R_{\mu + \nu + 2k}^* (B_{\mu + \nu + 2k})^{-1} E_{\mu,k}^{\nu}(f).
\end{flalign*}
Now this expression is only symbolic. Note that in Proposition \ref{domainok} we prove that for $f \in L^2(\D, d \iota)$ the function $E_{\mu,k}^{\nu}(f)$ lies in the domain of $B_{\mu + \nu + 2k}^{-1}$, and it follows that indeed

$$ \mathcal{T}_{\mu,k}^{\nu}(R_{\mu}^*(f)) = R_{\mu + \nu + 2k}^* (B_{\mu + \nu + 2k})^{-1} E_{\mu,k}^{\nu}(f) \in L^2(\D, d \iota), $$
a well-defined quantity. This is also equal to
$$T_{ ((\mu + \nu + 2k) B_{\mu + \nu + 2k})^{-1} E_{\mu,k}^{\nu}(f)}$$
where $T_f$ is the Toeplitz operator of $f$ of Lemma \ref{Toeplitzrmu}.
As $(\mu + \nu + 2k) B_{\mu + \nu + 2k}$ goes to the identity strongly, the natural guess would be that
$$ \lim_{\nu \rightarrow \infty} \frac{1}{\nu} \Tr(\mathcal{T}_{\mu,k}^{\nu}(R_{\mu}^*(f))^n) = \int_{\D} (E_{\mu,k}(f))^n d \iota(z).$$
We make this rigorous in the next section.

\section{Trace formulas}
\label{tracfor}
We start with an elementary lemma.

\begin{lemma}
\label{abskernbdd}
For any integer $n \geq 1$ the integral
$$ I_n(\nu) \coloneqq (\nu - 1)^n \int_{\D^n}  \left\lvert \frac{(1 - |z_1|^2)^{\nu} \dots (1 - |z_n|^2)^{\nu}}{(1 - z_1 \li{z_2})^{\nu} \dots (1 - z_{n-1} \li{z_n})^{\nu}} \right\lvert d \iota(z_1) \dots d \iota(z_n)$$
is bounded in $\nu$. More precisely, for $\nu \geq 4$ we have $I_n(\nu) \leq 3^{2n}$.
\end{lemma}

\begin{proof}
For $n = 1$ the integral simplifies to
$$ (\nu - 1) \int_{\D} (1 - |z_1|^2)^{\nu} d \iota(z_1) = \int_{\D} d \iota_{\nu}(z_1) = 1.$$
Now first we assume $\nu \geq 4$ is even, i.e. $\nu = 2 \kappa$. Then the integral is
\begin{flalign*}
& I_n(\nu) = \int_{\D^n} \left\lvert \frac{1}{(1 - z_1 \li{z_2}) \dots (1 - z_{n-1} \li{z_n})} \right\lvert^{2 \kappa} d \iota_{\nu}(z_1) \dots d \iota_{\nu}(z_n)
\\ & = \nm{ \sum_{i_1, \dots, i_{n-1}} \frac{ (\kappa)_{i_1}}{i_1 !} \dots \frac{ (\kappa)_{i_{n-1}}}{i_{n-1} !} z_1^{i_1} z_2^{i_1 + i_2} \dots z_{n-1}^{i_{n-1}}}^2_{\Hc_{\nu}^{\bigotimes n-1}}
\\ & = \sum_{i_1, \dots, i_{n-1}} ( \frac{(\kappa)_{i_1}}{i_1 !} \dots \frac{(\kappa)_{i_{n-1}}}{i_{n-1} !})^2 \frac{i_1 !}{(2 \kappa)_{i_1}} \frac{(i_1 + i_2) !}{(2 \kappa)_{i_1 + i_2}} \dots \frac{(i_{n-2} + i_{n-1}) !}{(2 \kappa)_{i_{n-2} + i_{n-1}}} \frac{i_{n-1} !}{(2 \kappa)_{i_{n-1}}}.
\end{flalign*}
Now we need the following inequality
\begin{equation}
\label{combineq}
3 \frac{j!}{(\kappa)_j} \geq \frac{2 \kappa - 1}{\kappa - 1} \frac{j!}{(\kappa)_j} = \sum_{i=0}^{\infty} \frac{ (\kappa)_i}{i!} \frac{(i+j)!}{(2 \kappa)_{i+j}}.
\end{equation}
We prove this. The inequality is obvious, for the equality we get
\begin{flalign*}
& \frac{2 \kappa - 1}{\kappa - 1} \frac{j!}{(\kappa)_j} = \frac{2 \kappa - 1}{\kappa - 1} \nm{z^j}_{\kappa}^2 = \frac{2 \kappa - 1}{\kappa - 1} \int_{\D} |z|^{2j} d \iota_{\kappa}(z)
\\ & = (2 \kappa - 1) \int_{\D} |z|^{2j}(1 - |z|^2)^{\kappa} d \iota(z) = (2 \kappa - 1) \int_{\D} \frac{|z|^{2j} (1 - |z|^2)^{2 \kappa}}{(1 - |z|^2)^{\kappa}} d \iota(z)
\\ & = \int_{\D} \sum_{i=0}^{\infty} \frac{(\kappa)_i}{i!} |z|^{2(i+j)} d \iota_{2 \kappa}(z) = \sum_{i=0}^{\infty} \frac{(\kappa)_i}{i!} \nm{z^{i+j}}_{2 \kappa} = \sum_{i=0}^{\infty} \frac{(\kappa)_i}{i!} \frac{(i+j)!}{(2 \kappa)_{i+j}}.
\end{flalign*}
Then for our sum
$$\sum_{i_1, \dots, i_{n-1}} ( \frac{(\kappa)_{i_1}}{i_1 !} \dots \frac{(\kappa)_{i_n}}{i_n !})^2 \frac{i_1 !}{(2 \kappa)_{i_1}} \frac{(i_1 + i_2) !}{(2 \kappa)_{i_1 + i_2}} \dots \frac{(i_{n-2} + i_{n-1}) !}{(2 \kappa)_{i_{n-2} + i_{n-1}}} \frac{i_{n-1} !}{(2 \kappa)_{i_{n-1}}}$$
we get that by inequality (\ref{combineq}) the part summing over $i_1$ is
\begin{flalign*}
& \sum_{i_1 = 0}^{\infty} ( \frac{(\kappa)_{i_1}}{i_1 !})^2 \frac{i_1 !}{(2 \kappa)_{i_1}} \frac{(i_1 + i_2) !}{(2 \kappa)_{i_1 + i_2}} \leq \sum_{i_1 = 0}^{\infty} \frac{(\kappa)_{i_1}}{i_1 !} \frac{i_1 !}{(2 \kappa)_{i_1}} \sum_{i_1 = 0}^{\infty} \frac{(\kappa)_{i_1}}{i_1 !} \frac{(i_1 + i_2) !}{(2 \kappa)_{i_1 + i_2}}
\\ & \leq 9 \frac{0!}{(\kappa)_0} \frac{i_2 !}{(\kappa)_{i_2}} = 9 \frac{i_2 !}{(\kappa)_{i_2}}.
\end{flalign*}
Thus for our sum we get
\begin{flalign*}
& \sum_{i_1, \dots, i_{n-1}} ( \frac{(\kappa)_{i_1}}{i_1 !} \dots \frac{(\kappa)_{i_{n-1}}}{i_{n-1} !})^2 \frac{i_1 !}{(2 \kappa)_{i_1}} \frac{(i_1 + i_2) !}{(2 \kappa)_{i_1 + i_2}} \dots \frac{(i_{n-2} + i_{n-1}) !}{(2 \kappa)_{i_{n-2} + i_{n-1}}} \frac{i_{n-1} !}{(2 \kappa)_{i_{n-1}}}
\\ & \leq 9 \sum_{i_2, \dots, i_{n-1}} \frac{(\kappa)_{i_2}}{i_2 !} ( \frac{(\kappa)_{i_3}}{i_3 !} \dots \frac{(\kappa)_{i_{n-1}}}{i_{n-1} !})^2 \frac{(i_2 + i_3) !}{(2 \kappa)_{i_2 + i_3}} \dots \frac{(i_{n-2} + i_{n-1}) !}{(2 \kappa)_{i_{n-2} + i_{n-1}}} \frac{i_{n-1} !}{(2 \kappa)_{i_{n-1}}}.
\end{flalign*}
Repeatedly applying Inequality \ref{combineq} this is smaller than
\begin{flalign*}
& 3^{n-1} \sum_{i_{n-1} = 0}^{\infty} \frac{(\kappa)_{i_{n-1}}}{i_{n-1} !} \frac{i_{n-1} !}{(2 \kappa)_{i_{n-1}}} \leq 3^n \frac{0!}{(\kappa)_0} = 3^n.
\end{flalign*}
Now we look at what happens when $\nu$ is odd. Using
$$|1 - x \li{y}|^2 \geq (1 - |x|^2)(1 - |y|^2)$$
we see
\begin{flalign*}
& I_n(\nu) = (\nu - 1)^n \int_{\D^n}  \left\lvert \frac{(1 - |z_1|^2)^{\nu} \dots (1 - |z_n|^2)^{\nu}}{(1 - z_1 \li{z_2})^{\nu} \dots (1 - z_{n-1} \li{z_n})^{\nu}}  \right\rvert d \iota(z_1) \dots d \iota(z_n)
\\ & \leq (\nu - 2)^n (\frac{\nu-1}{\nu - 2})^n \int_{\D^n}  \left\lvert \frac{(1 - |z_1|^2) \dots (1 - |z_n|^2)}{(1 - z_1 \li{z_2}) \dots (1 - z_{n-1} \li{z_n})}  \right\rvert^{\nu - 1} d \iota(z_1) \dots d \iota(z_n)
\\ & \leq 3^n (\frac{\nu-1}{\nu - 2})^n \leq 3^{2n}.
\end{flalign*}
We have now proven our result, i.e. given a bound independent of $\nu$.
\end{proof}

We now prove the following lemma.

\begin{lemma}
\label{tracexncpt}
For any $n \geq 1$ and $f \in C_c(\D)$
\begin{flalign*}
& \lim_{\nu \rightarrow \infty} \frac{1}{\nu-1} \Tr(( (\nu - 1) R_{\nu}^*(f) )^n) = \int_{\D} f(z)^n d \iota(z).
\end{flalign*}
\end{lemma}

\begin{proof}
We note

\begin{flalign*}
& R_{\nu}^*(f)^n(x,y)
\\ & = \int_{\D^n} \frac{K^{\nu}(x,x_1) K^{\nu}(x_1, x_2) \dots K^{\nu}(x_n, y) f(x_1) \dots f(x_n)}{K^{\nu}(x_1,x_1) \dots K^{\nu}(x_n, x_n)} d \iota(x_1) \dots d \iota(x_n).
\end{flalign*}
Next we see that

\begin{flalign*}
& \frac{1}{\nu-1} \Tr(( (\nu - 1) R_{\nu}^*(f) )^n) = (\nu-1)^{n} \int_{\D} R_{\nu}(R_{\nu}^*(f)^n)(z) d \iota(z)
\\ & = (\nu - 1)^{n} \int_{\D^{n+1}} \frac{K^{\nu}(z,x_1) \dots K^{\nu}(x_n,z) f(x_1) \dots f(x_n)}{K^{\nu}(z,z) K^{\nu}(x_1, x_1) \dots K^{\nu}(x_n, x_n)} d \iota(x_1) \dots \iota(x_{n}) d \iota(z)
\\ & = (\nu - 1)^{n-1} \int_{\D^{n+1}} \frac{K^{\nu}(z,x_1) \dots K^{\nu}(x_n,z) f(x_1) \dots f(x_n)}{K^{\nu}(x_1, x_1) \dots K^{\nu}(x_n, x_n)} 
\\ & d \iota_{\nu}(z) d \iota(x_1) \dots d \iota(x_{n})
\\ & = (\nu - 1)^{n-1} \int_{\D^n} \frac{K^{\nu}(x_1, x_2) \dots K^{\nu}(x_n,x_1) f(x_1) \dots f(x_n)}{K^{\nu}(x_1, x_1) \dots K^{\nu}(x_n, x_n)}  d \iota(x_1) \dots d \iota(x_{n})
\\ & = (\nu - 1)^{n-1} \int_{\D} f(x_1) \int_{\D^{n-1}} \frac{K^{\nu}(x_1, x_2) \dots K^{\nu}(x_n,x_1) f(x_2) \dots f(x_n)}{K^{\nu}(x_1, x_1) \dots K^{\nu}(x_n, x_n)} \\ & d \iota(x_2) \dots \iota(x_{n}) d \iota(x_1)
\end{flalign*}
Now we study
\begin{equation}
\label{ptwiseconving}
(\nu - 1)^{n-1} \int_{\D^{n-1}} \frac{K^{\nu}(x_1, x_2) \dots K^{\nu}(x_n,x_1) f(x_2) \dots f(x_n)}{K^{\nu}(x_1, x_1) \dots K^{\nu}(x_n, x_n)} d \iota(x_2) \dots \iota(x_{n}),
\end{equation}
which is equal to $(\nu - 1)^{n-1} R_{\nu}(R_{\nu}^*(f)^{n-1})(x_1)$. We will prove
\begin{equation}
\label{limitptwise}
\lim_{\nu \rightarrow \infty} (\nu - 1)^{n-1} R_{\nu}(R_{\nu}^*(f)^{n-1})(x_1) = f(x_1)^{n-1}
\end{equation}
for any $x_1 \in \D$. Now the function $f$ is compactly supported and
$$(\nu - 1)^{n-1}R_{\nu}(R_{\nu}^*(f)^{n-1})$$
is continuous by Definition \ref{Rmudef} and bounded independent of $\nu$ by Proposition \ref{factsRmu} and Lemma \ref{Toeplitzrmu}, so Equation (\ref{limitptwise}) implies
$$\lim_{\nu \rightarrow \infty} \frac{1}{\nu} \Tr(R_{\nu}^*(f)^n) = \int_{\D} f(z)^n d \iota (z).$$
We now prove pointwise convergence of Equation (\ref{ptwiseconving}) where we replace $n-1$ by $n$ for the sake of convenience. Note that for $z = g \cdot 0 \in \D$

\begin{flalign*}
& R_{\nu}(R_{\nu}^*(f)^n)(z) = (g^{-1} R_{\nu} (R_{\nu}^*(f)^n))(0) = R_{\nu}(R_{\nu}^*(g^{-1}f)^n)(0).
\end{flalign*}
We conclude it is enough to prove convergence in the point $0$. Observe

\begin{flalign*}
& (\nu - 1)^{n} R_{\nu}(R_{\nu}^*(f)^n)(0)
\\ & = (\nu - 1)^n \int_{\D^n} K^{\nu}(0,z_1) K^{\nu}(z_1, z_2) \dots K^{\nu}(z_n,0) (1 - |z_1|^2)^{\nu} \dots (1 - |z_n|^2)^{\nu} \cdot 
\\ & f(z_1) \dots f(z_n) d \iota(z_1) \dots d \iota(z_n)
\\ & = (\nu - 1)^n \int_{\D^n} \frac{(1 - |z_1|^2)^{\nu} \dots (1 - |z_n|^2)^{\nu}}{(1 - z_1 \li{z_2})^{\nu} \dots (1 - z_{n-1} \li{z_n})^{\nu}} f(z_1) \dots f(z_n) d \iota(z_1) \dots d \iota(z_n).
\end{flalign*}
Now we choose $(0, \dots, 0) \in U \subseteq \D^n$ open such that for $(x_1, \dots, x_n) \in U$
$$|f(x_1) \dots f(x_n) - f(0)^n| < \epsilon.$$
Using Lemma \ref{abskernbdd} it follows that
\begin{flalign*}
& |(\nu - 1)^{n} R_{\nu}(R_{\nu}^*(f)^n)(0) - f(0)^n|
\\ & = (\nu - 1)^n | \int_{\D^n} \frac{(1 - |z_1|^2)^{\nu} \dots (1 - |z_n|^2)^{\nu}}{(1 - z_1 \li{z_2})^{\nu} \dots (1 - z_{n-1} \li{z_n})^{\nu}} (f(z_1) \dots f(z_n) - f(0)^n)
\\ & d \iota(z_1) \dots d \iota(z_n)|
\\ & \leq (\nu - 1)^n \int_{U} \left\lvert \frac{(1 - |z_1|^2)^{\nu} \dots (1 - |z_n|^2)^{\nu}}{(1 - z_1 \li{z_2})^{\nu} \dots (1 - z_{n-1} \li{z_n})^{\nu}} (f(z_1) \dots f(z_n) - f(0)^n)  \right\rvert
\\ & d \iota(z_1) \dots d \iota(z_n)
\\ & + (\nu - 1)^n \int_{\D^n \backslash U} \left\lvert \frac{(1 - |z_1|^2)^{\nu} \dots (1 - |z_n|^2)^{\nu}}{(1 - z_1 \li{z_2})^{\nu} \dots (1 - z_{n-1} \li{z_n})^{\nu}} (f(z_1) \dots f(z_n) - f(0)^n)  \right\rvert
\\ & d \iota(z_1) \dots d \iota(z_n)
\\ & \leq \epsilon 3^{2n}
\\ & + (\nu - 1)^n \int_{\D^n \backslash U} \left\lvert \frac{(1 - |z_1|^2)^{\nu} \dots (1 - |z_n|^2)^{\nu}}{(1 - z_1 \li{z_2})^{\nu} \dots (1 - z_{n-1} \li{z_n})^{\nu}} (f(z_1) \dots f(z_n) - f(0)^n) \right\rvert
\\ & d \iota(z_1) \dots d \iota(z_n).
\end{flalign*}
Observe that $|1 - x \li{y}|^2 \geq (1 - |x|^2)(1 - |y|^2)$ with equality if and only if $x = y$, so for $(z_1, \dots z_n) \in \D^n \backslash U$
$$ | \frac{(1 - |z_1|^2) \dots (1 - |z_n|^2)}{(1 - z_1 \li{z_2}) \dots (1 - z_{n-1} \li{z_n})} | \leq \sqrt{1-|z_1|^2} \sqrt{1-|z_n|^2} < r < 1$$
for some $r \in \R$. 
Thus Lebesgue's dominated convergence theorem implies
\begin{flalign*}
& \lim_{\nu \rightarrow \infty} (\nu - 1)^n \int_{\D^n \backslash U} | \frac{(1 - |z_1|^2)^{\nu} \dots (1 - |z_n|^2)^{\nu}}{(1 - z_1 \li{z_2})^{\nu} \dots (1 - z_{n-1} \li{z_n})^{\nu}} (f(z_1) \dots f(z_n) - f(0)^n)|
\\ & d \iota(z_1) \dots d \iota(z_n) = 0.
\end{flalign*}
It follows that for any $x_1 \in \D$
$$ \lim_{\nu \rightarrow \infty} R_{\nu}(R_{\nu}^*(f)^n)(x_1) =  f(x_1)^n,$$
which proves the theorem.
\end{proof}

We prove another lemma.

\begin{lemma}
\label{tracexn}
For $f \in L^2(\D, d \iota)$ bounded we have
\begin{flalign*}
& \lim_{\nu \rightarrow \infty} \frac{1}{\nu-1} \Tr(( (\nu - 1) R_{\nu}^*(f) )^{2n}) = \int_{\D} f(z)^{2n} d \iota(z)
\end{flalign*}
for $n \geq 1$.
\end{lemma}

\begin{proof}
Let $f \in L^2(\D, d \iota)$ be bounded. Then $f \in L^p(\D, d \iota)$ for any $p > 2$, implying that
$$ \int_{\D} f(z)^{2n} d \iota(z)$$
is well-defined. We claim that there is a sequence $\{f_k\}_k$ of compactly supported functions such that
$$ \lim_{k \rightarrow \infty} \nm{f - f_k}_2 = 0$$
and $\nm{f_k}_{\infty} \leq \nm{f}_{\infty}$ for all $k$. The way to see this is the following. Let us denote $f|_A = \chi_A f$, where $\chi_A$ is te indicator function. Let $k$ be an integer, then we have a finite measure set $E_k \subseteq \D$ such that
$$\nm{f|_{E_k} - f}_2 \leq \frac{1}{k}.$$
Then by Lusin's theorem \cite[Theorem 7.10]{foll} there exists a function $f_k \in C_c(\D)$ such that $\nm{f_k}_{\infty} \leq \nm{f|_{E_k}}_{\infty} \leq \nm{f}_{\infty}$ and
$$ \mu(\{ f|_{E_k} \neq f_k \}) \leq \frac{1}{k^2}.$$
Thus $\nm{f - f_k}_2 \leq (1 + 2 \nm{f}_{\infty})\frac{1}{k},$
which proves this claim.

Now we write $T_f$ for $(\nu - 1) R_{\nu}^*(f)$ and we get by Proposition \ref{Rmubdd}
$$ \frac{1}{\sqrt{\nu - 1}} \nm{T_f - T_{f_k} }_{2} \leq \nm{f - f_k}_2,$$
and thus
$$ \lim_{\nu \rightarrow \infty} \frac{1}{\sqrt{\nu - 1}} \nm{T_f - T_{f_k} }_{2} = 0,$$
Now we prove by induction for any integer $n$
$$ \lim_{k \rightarrow \infty} \frac{1}{\sqrt{\nu - 1}} \nm{T_f^n - T_{f_k}^n }_{2} = 0,$$
with rate of convergence independent of $\nu$, for any $n \in \N$. The case $n = 1$ is clear, now let $n > 0$.
\begin{flalign*}
& \frac{1}{\sqrt{\nu - 1}} \nm{T_f^{n+1} - T_{f_k}^{n+1}}
\\ & \leq \frac{1}{\sqrt{\nu - 1}} (\nm{T_f ( T_f^{n} - T_{f_k}^n)}_{2} + \nm{(T_f - T_{f_k}) T_{f_k}^n}_{2})
\\ & \leq \frac{1}{\sqrt{\nu - 1}} \nm{T_f} \cdot \nm{T_f^{n} - T_{f_k}^n}_{2} + \frac{1}{\sqrt{\nu - 1}} \nm{T_f - T_{f_k}}_{2} \cdot \nm{T_{f_k}^n}
\\ & \leq \frac{1}{\sqrt{\nu - 1}} \nm{f}_{\infty} \nm{T_f^{n} - T_{f_k}^n}_{2} + \frac{1}{\sqrt{\nu - 1}} \nm{T_f - T_{f_k}}_{2} \cdot \nm{f_k}_{\infty}^n
\\ & \leq \frac{1}{\sqrt{\nu - 1}} \nm{f}_{\infty} \nm{T_f^{n} - T_{f_k}^n}_{2} + \frac{1}{\sqrt{\nu - 1}} \nm{T_f - T_{f_k}}_{2} \nm{f}_{\infty}^n,
\end{flalign*}
where we have used that $\nm{T_f} = \nm{P M_f P} \leq \nm{f}_{\infty}$. Thus indeed for all $n \in \N$
$$ \lim_{k \rightarrow \infty} \frac{1}{\sqrt{\nu - 1}} \nm{T_f^n - T_{f_k}^n }_{2} = 0,$$
and the rate of convergence is independent of $\nu$. Thus by explicit realization of the Hilbert Schmidt norm
\begin{flalign*}
& \lim_{k \rightarrow \infty} \frac{1}{\nu - 1} \Tr(((\nu - 1) R_{\nu}^*(f_k))^{2n}) - \frac{1}{\nu - 1} \Tr(((\nu - 1) R_{\nu}^*(f))^{2n})
\\ & = \lim_{k \rightarrow \infty} \frac{1}{\nu - 1} \langle ((\nu - 1)R_{\nu}^*(f_k))^n, ((\nu - 1) R_{\nu}^*(f_k)^*)^n \rangle_{S^2(\Hc_{\mu})}
\\ & - \frac{1}{\nu - 1} \langle ((\nu - 1)R_{\nu}^*(f))^{n}, ((\nu - 1)R_{\nu}^*(f)^*)^{n} \rangle_{S^2(\Hc_{\mu})} = 0.
\end{flalign*}
for any $\nu$. Furthermore, for any $p > 2$
$$ \lim_{k \rightarrow \infty} \nm{f-f_k}_p = 0.$$
Thus
\begin{flalign*}
& |\frac{1}{\nu-1} \Tr(( (\nu - 1) R_{\nu}^*(f) )^{2n}) - \int_{\D} f(z)^{2n} d \iota(z)|
\\ & \leq |\frac{1}{\nu-1} \Tr(( (\nu - 1) R_{\nu}^*(f) )^{2n}) - \frac{1}{\nu-1} \Tr(( (\nu - 1) R_{\nu}^*(f_k) )^{2n})|
\\ & + |\frac{1}{\nu-1} \Tr(( (\nu - 1) R_{\nu}^*(f_k) )^{2n}) - \int_{\D} f_k(z)^{2n} d \iota(z)|
\\ & + |\int_{\D} f_k(z)^{2n} d \iota(z) - \int_{\D} f(z)^{2n} d \iota(z)|.
\end{flalign*}
We can choose $k$ first to bound
$$|\frac{1}{\nu-1} \Tr(( (\nu - 1) R_{\nu}^*(f) )^{2n}) - \frac{1}{\nu-1} \Tr(( (\nu - 1) R_{\nu}^*(f_k) )^{2n})|$$
and
$$ |\int_{\D} f_k(z)^{2n} d \iota(z) - \int_{\D} f(z)^{2n} d \iota(z)| $$
independently of $\nu$, and then choose $N$ so that
$$|\frac{1}{\nu-1} \Tr(( (\nu - 1) R_{\nu}^*(f_k) )^{2n}) - \int_{\D} f_k(z)^{2n} d \iota(z)|$$
becomes arbitrarily small for $\nu \geq N$ by Lemma \ref{tracexncpt}. This proves our lemma.
\end{proof}

We recall that
$$ \mathcal{T}_{\mu,k}^{\nu}(R_{\mu}^*(f)) = R_{\mu + \nu + 2k}^* ( B_{\mu + \nu + 2k})^{-1} E_{\mu,k}^{\nu}(f),$$
so we also need to study
$$(B_{\mu + \nu + 2k})^{-1} E_{\mu,k}^{\nu}(f).$$
We do this in the next section.

\section{Inverse Berezin transform}
\label{inversebersec}


Our goal in this section is to prove the following.

\begin{proposition}
\label{ourinverseberezin}
We have
$$\lim_{\nu \rightarrow \infty} \nm{ ((\mu + \nu + 2k - 1)B_{\mu + \nu + 2k})^{-1} E_{\mu,k}^{\nu}(f) - E_{\mu,k}(f)}_{2} = 0.$$
\end{proposition}

We recall the expansions of $E_{\mu,k}^{\nu}$ and $E_{\mu,k}$ in terms of Berezin transforms in Proposition \ref{Elimit} and Definition \ref{noncptdefEmunuk}. Then Proposition \ref{ourinverseberezin} follows from the following proposition.

\begin{proposition}
\label{inverseberezin}
For $\nu_0$ fixed and $f \in L^2(\D, d \iota)$, we have convergence 
$$ \lim_{\nu \rightarrow \infty} \nm{((\nu - 1)B_{\nu})^{-1}B_{\nu_0}(f) - B_{\nu_0}(f)}_{2} = 0.$$
\end{proposition}

To prove this we need some bounds on the eigenvalues of $B_{\nu}^{-1} B_{\nu_0}$. We recall the definition of $b_{\nu}$ from Proposition \ref{berezineigen}.

\begin{lemma}
\label{bound1}
Let $\nu \geq \nu_0$ integers. The term $b_{\nu}^{-1}(\lambda) b_{\nu_0}(\lambda)$ is uniformly bounded in $\nu$ and $\lambda \in \R$.
\end{lemma}

\begin{proof}
We calculate
\begin{flalign*}
& |b_{\nu}(\lambda)^{-1} b_{\nu_0}(\lambda) | = | \frac{\Gamma(\nu) \Gamma(\nu - 1)}{\Gamma(\nu_0) \Gamma(\nu_0 - 1)} \prod_{k= \nu_0}^{\nu - 1} ( (k- \frac{1}{2})^2 + \lambda^2)^{-1}|
\\ & \leq | \frac{\Gamma(\nu) \Gamma(\nu - 1)}{\Gamma(\nu_0) \Gamma(\nu_0 - 1)} \prod_{k= \nu_0}^{\nu - 1} (k- \frac{1}{2})^{-2}| = |\frac{ \prod_{k=1}^{\nu_0 - 1} (k-\frac{1}{2})^{2}}{\Gamma(\nu_0) \Gamma(\nu_0 - 1)} \frac{\Gamma(\nu) \Gamma(\nu -1)}{\pi \Gamma(\nu - \frac{1}{2})^2}|.
\end{flalign*}
This is bounded uniformly in $\nu$, as
$$\lim_{\nu \rightarrow \infty} \frac{\Gamma(\nu) \Gamma(\nu -1)}{\pi \Gamma(\nu - \frac{1}{2})^2} = \frac{1}{\pi}.$$
\end{proof}

Now we prove that $ B_{\mu + \nu + 2k}^{-1} E_{\mu,k}^{\nu}(f)$ is well-defined.

\begin{proposition}
\label{domainok}
For $\nu_0 \leq \nu$, $B_{\nu_0}(L^2(\D, d \iota))$ is in the domain of $((\nu-1)B_{\nu})^{-1}$.
\end{proposition}

\begin{proof}
Let $f \in L^2(\D, d \iota)$. Then by Proposition \ref{helginv}
$$\tilde{f} \in L^2(\R^+, |c(\lambda)|^{-2}d \lambda).$$
Now we see by Lemma \ref{bound1} that $b_{\nu}^{-1}(\lambda) b_{\nu_0}(\lambda)$ is bounded, hence $b_{\nu}^{-1} b_{\nu_0} \tilde{f}$ is in $L^2(\R^+ \times S^1, |c(\lambda)|^{-2}d \lambda)$ and by the Plancherel theorem, Proposition \ref{helginv}, we can define $F \in L^2(\D, d \iota)$ such that $F^{\sim} = b_{\nu}^{-1} b_{\nu_0} \tilde{f}$. Applying $(\nu-1) B_{\nu}$ we see
$$(\nu-1) B_{\nu}(F) = (\nu_0 - 1) B_{\nu_0}(f).$$
Hence $B_{\nu_0}(L^2(\D, d \iota))$ is in the domain of $((\nu-1)B_{\nu})^{-1}$.
\end{proof}

Next we prove Proposition \ref{inverseberezin}, which makes sense now.

\begin{proof}
We study
$$\lim_{\nu \rightarrow \infty} \nm{((\nu - 1)B_{\nu})^{-1} B_{\nu_0}(f) - B_{\nu_0}(f)}_2.$$
We notice that for $f \in L^2(\D, d \iota)$ 
$$(\nu - 1)B_{\nu}(f)^{\sim}(\lambda,b) = b_{\nu}(\lambda) \tilde{f}(\lambda,b),$$
and for $\nu \geq \nu_0$
$$((\nu - 1)B_{\nu})^{-1} B_{\nu_0}(f)^{\sim}(\lambda,b) = \frac{b_{\nu}^{-1}(\lambda) b_{\nu_0}}{\nu_0 - 1}(\lambda) \tilde{f}(\lambda,b).$$ 
Thus we see by the Plancherel formula in Proposition \ref{helginv}

\begin{flalign*}
& \nm{((\nu - 1)B_{\nu})^{-1} B_{\nu_0}(f) - B_{\nu_0}(f)}_2^2
\\ & = \frac{1}{2 \pi^2} \int_{\R} \int_{S^1} |\tilde{f}(\lambda,b)|^2 \left\lvert \frac{b_{\nu}(\lambda)^{-1} b_{\nu_0}(\lambda) - b_{\nu_0}(\lambda)}{\nu_0 - 1} \right\rvert^2 |c(\lambda)|^{-2} db d \lambda.
\end{flalign*}
If $B \in \R_{>0}$ and $\epsilon > 0$, then by Lemma \ref{boundsbnulambda} there exists an $N$ such that if $| \lambda| < B$ and $\nu \geq N$, then

$$|b_{\nu}(\lambda) - 1| < \epsilon.$$
Now we more specifically look at the term
$$|b_{\nu}(\lambda)^{-1} b_{\nu_0}(\lambda) - b_{\nu_0}(\lambda)|.$$
We want to use this to bound the term

\begin{flalign*}
& \nm{((\nu - 1)B_{\nu})^{-1} B_{\nu_0}(f) - B_{\nu_0}(f)}_2^2
\\ & = \frac{1}{2 \pi^2} \int_{\R} \int_{S^1} |\tilde{f}(\lambda,b)|^2 |b_{\nu}(\lambda)^{-1} b_{\nu_0}(\lambda) - b_{\nu_0}(\lambda)|^2 |c(\lambda)|^{-2} db d \lambda.
\end{flalign*}
We first note that by Lemma \ref{bound1} $\left\lvert \frac{b_{\nu}(\lambda)^{-1} b_{\nu_0}(\lambda) - b_{\nu_0}(\lambda)}{\nu_0 - 1} \right\rvert$ is uniformly bounded in $\nu$ and $\lambda$, say by $R > 0$. Next we note that by the Plancherel formula in Proposition \ref{helginv}

$$\frac{1}{2 \pi^2} \int_{\R} \int_{S^1} |\tilde{f}(\lambda, b)|^2 | c( \lambda)|^{-2} db d \lambda = \nm{\tilde{f}}_2^2 = \nm{f}_2^2 < \infty.$$
Now let $\epsilon > 0$. Choose $B > 0$ such that

$$ \frac{1}{2 \pi^2} \int_{\R \backslash [-B,B]} \int_{S^1} |\tilde{f}(\lambda, b)|^2 |c(\lambda)|^{-2} db d \lambda < \frac{\epsilon}{R^2}.$$
Now choose $N$ such that for $\nu \geq N$ and $|\lambda| \leq B$

$$ \left\lvert \frac{b_{\nu}(\lambda)^{-1} b_{\nu_0}(\lambda) - b_{\nu_0}(\lambda)}{\nu_0 - 1} \right\rvert^2 < \epsilon.$$
Then for $\nu \geq N$

\begin{flalign*}
& \nm{((\nu - 1)B_{\nu})^{-1} B_{\nu_0} (f) - B_{\nu_0}(f)}_2^2
\\ & = \frac{1}{2 \pi^2} \int_{\R} \int_{S^1} |\tilde{f}(\lambda,b)|^2 \left\lvert \frac{b_{\nu}(\lambda)^{-1} b_{\nu_0}(\lambda) - b_{\nu_0}(\lambda)}{\nu_0 - 1} \right\rvert^2 |c(\lambda)|^{-2} db d \lambda
\\ & = \frac{1}{2 \pi^2} \int_{[-B,B]} \int_{S^1} |\tilde{f}(\lambda,b)|^2 \left\lvert \frac{b_{\nu}(\lambda)^{-1} b_{\nu_0}(\lambda) - b_{\nu_0}(\lambda)}{\nu_0 - 1} \right\rvert^2 |c(\lambda)|^{-2} db d \lambda
\\ & + \frac{1}{2 \pi^2} \int_{\R \backslash [-B,B]} \int_{S^1} |\tilde{f}(\lambda,b)|^2 \left\lvert \frac{b_{\nu}(\lambda)^{-1} b_{\nu_0}(\lambda) - b_{\nu_0}(\lambda)}{\nu_0 - 1} \right\rvert^22 |c(\lambda)|^{-2} db d \lambda
\\ & \leq \frac{\epsilon}{2 \pi^2} \int_{[-B,B]} \int_{S^1} |\tilde{f}(\lambda,b)|^2 |c(\lambda)|^{-2} db d \lambda
\\ & + \frac{R^2}{2 \pi^2} \int_{\R \backslash [-B,B]} \int_{S^1} |\tilde{f}(\lambda,b)|^2  |c(\lambda)|^{-2} db d \lambda \leq \epsilon \nm{f}_2^2 + \frac{\epsilon}{2 \pi^2}.
\end{flalign*}
This completes the proof.
\end{proof}

\section{Trace calculation}
\label{tracecalc}

Now we want to draw conclusions about the limit of the trace of the functional calculus

$$ \Tr(\mathcal{T}_{\mu,k}^{\nu}(R_{\mu}(f))^n) = \Tr((R_{\mu + \nu + 2k}^* (B_{\mu + \nu + 2k})^{-1} E_{\mu,k}^{\nu}(f))^n). $$
We know that for $f \in L^2(\D, d \iota)$
$$ \lim_{\nu \rightarrow \infty} \nm{( (\mu + \nu + 2k -1) B_{\mu + \nu + 2k})^{-1} E_{\mu,k}^{\nu}(f) - E_{\mu,k}(f)}_2 = 0$$
by Proposition \ref{ourinverseberezin}. We need an additional bound.

\begin{lemma}
\label{boundcptfctns}
For $f \in C_c^{\infty}(\D)$ and $\nu_0$ fixed there is some $R > 0$ such that
$$\nm{((\nu - 1)B_{\nu})^{-1} (\nu_0-1) B_{\nu_0}(f)} \leq R$$
for all $\nu \geq \nu_0$.
\end{lemma}

\begin{proof}
First we let
$$ \phi_{n, \lambda}(z) = \int_{S^1} e_{\lambda,b}(z) b^n d b.$$
We see
$$ | \phi_{n, \lambda}(z)| \leq \int_{S^1} e_{0,b}(z) dk = \phi_{0}(z) \leq 1,$$
so this is bounded. Note $\phi_0$ is the Harish-Chandra $\Xi$-function. Using Fourier expansion on the circle we write
$$\tilde{f}(\lambda,b) = \sum_{n= - \infty}^{\infty} \tilde{f}(\lambda,n) b^n$$
for $\lambda \in \R$ and $b \in S^1$. Thus we see by Proposition \ref{helginv}
\begin{flalign*}
& f(z) =  \int_{\R} \int_{S^1} \tilde{f}(- \lambda, b) e_{\lambda,b}(z) |c(\lambda)|^{-2} db d \lambda
\\ & = \int_{\R} \int_{S^1} \sum_{n = - \infty}^{\infty} \tilde{f}(-\lambda,n) b^n e_{\lambda,b}(z) |c(\lambda)|^{-2} db d \lambda
\\ & = \int_{\R} \sum_{n = - \infty}^{\infty} \int_{S^1} \tilde{f}(-\lambda,n) b^n e_{\lambda,b}(z) |c(\lambda)|^{-2} db d \lambda
\\ & = \int_{\R} \sum_{n = - \infty}^{\infty}\tilde{f}(-\lambda,n) \phi_{n,\lambda}(g) d \mu( \lambda).
\end{flalign*}
Note that if
$$\tilde{g}(z,e^{is}) \coloneqq \frac{d^2}{dt^2}|_{t=s} \tilde{f}(\lambda, e^{it})$$
then for the Fourier transform
$$ \tilde{g}(\lambda,n) = -n^2 \tilde{f}(\lambda,n).$$
Now we define
$$ F(\mu, b) = \int_{\R} \tilde{f}(\lambda, b) e^{i \mu \lambda} d \lambda.$$
Then by the Euclidean Paley-Wiener Theorem this is compactly supported and
$$ \tilde{f}(\lambda,b) = \int_{\R} F(\mu,b) e^{-i \mu \lambda} d \mu.$$
Thus $G(\mu,e^{is}) = \frac{d^2}{dt^2}|_{t=s} F(\mu,e^{it})$ is still compactly supported and
$$\tilde{g}(\lambda,b) = \int_{\R} G(\mu,b) e^{-i \mu \lambda} d \mu$$
has the property
$$(- i \lambda)^n \tilde{g}(\lambda,b) =  \int_{\R} (\frac{d^n}{d \mu^n}|_{\mu = \nu} G(\mu,b)) e^{- i \mu \lambda } d \nu.$$
This implies that for each $n \in \N$ there is some constant $C_n$ such that for all $\mu \in \R$ and $b \in S^1$
$$|\lambda^n \tilde{g}(\lambda,b) | \leq C_n,$$
so for each $N \in \N$
$$|\tilde{g}(\lambda,n)| = |\int_{S^1} \tilde{g}(\lambda, b) b^{-n} dk| \leq \int_{S^1} C (1 + |\lambda|)^{-N} d k = C_N' (1 + |\lambda|)^{-N},$$
for some new constant $C'_N$. This implies
$$|\tilde{f}(\lambda,n)| \leq C_N' \frac{1}{n^2} (1 + |\lambda|)^{-N}.$$
Using Lemma \ref{bound1} we obtain
\begin{flalign*}
& | \int_{\R} \int_{S^1} b_{\nu}(\lambda)^{-1} b_{\nu_0}(\lambda)  \tilde{f}(\lambda, b) e_{\lambda,b}(z) |c(\lambda)|^{-2} db d \lambda |
\\ & = | \int_{\R} \sum_{n = - \infty}^{\infty} b_{\nu}(\lambda)^{-1} b_{\nu_0}(\lambda) \tilde{f}(\lambda,n) \phi_{n,\lambda}(z) |c(\lambda)|^{-2} d \lambda|
\\ & \leq \int_{\R} \sum_{n = - \infty}^{\infty} | b_{\nu}(\lambda)^{-1} b_{\nu_0} \tilde{f}(\lambda,n) \phi_{n,\lambda}(z) |c(\lambda)|^{-2}|d \lambda
\\ & \leq R \int_{\R} \sum_{n = - \infty}^{\infty} C_N' \frac{1}{n^2} (1 + |\lambda|)^{-N}  |c(\lambda)|^{-2} d \lambda < B,
\end{flalign*}
for some constants $R$, $B$. Together with the fact that for some function $h \in L^2(\D, d \iota)$ such that $\tilde{h} \in L^1(\R^+ \times S^1, |c(\lambda)|^{-2} db d \lambda ) \cap L^2(\R^+ \times S^1, |c(\lambda)|^{-2} db d \lambda )$ we have
$$h(z) = \int_{\R} \int_{S^1} \tilde{h}(- \lambda, b) e_{\lambda,b}(z) |c(\lambda)|^{-2} db d \lambda$$
almost everywhere, this proves our lemma.
\end{proof}

Now we can prove the following for even integers.

\begin{lemma}
\label{convx2n}
Let $f \in L^2(\D, d \iota) \cap L^{\infty}(\D, d \iota)$ and $\{f_{\nu}\}_{\nu}$ a sequence such that
$$ \lim_{\nu \rightarrow \infty} \nm{f - f_{\nu}}_2 = 0$$
and there is some $B$ such that for all $\nu$
$$ \nm{f_{\nu}}_{\infty} \leq B.$$
Then for any integer $n$
$$ \lim_{\nu \rightarrow \infty} \frac{1}{\nu} \Tr(( (\nu - 1)R_{\nu}^*(f_{\nu}))^{2n}) = \int_{\D} f(z)^{2n} d \iota(z).$$
\end{lemma}

\begin{proof}
Using the notation from Lemma \ref{Toeplitzrmu} we claim
$$ \lim_{\nu \rightarrow \infty} \frac{1}{\sqrt{\nu}} \nm{T_{f_{\nu}}^{n} - T_{f}^{n}}_{2} = 0.$$
We prove this by induction. The case $n=0$ is trivial and the case $n=1$ follows from Proposition \ref{Rmubdd}. For the induction step we see
\begin{flalign*}
& \frac{1}{\sqrt{\nu}} \nm{ T_{f_{\nu}}^{n+1} - T_{f}^{n+1}}_{2}
\\ & \leq \frac{1}{\sqrt{\nu}} \nm{ (T_{f_{\nu}} - T_{f} ) T_{f_{\nu}}^{n}}_{2} + \frac{1}{\sqrt{\nu}} \nm{T_{f} (T_{f_{\nu}}^{n} - T_{f}^{n})}_{2}
\\ & \leq \frac{1}{\sqrt{\nu}} \nm{ T_{f_{\nu}} - T_{f}}_{2} \nm{T_{f_{\nu}}^{n}} + \frac{1}{\sqrt{\nu}} \nm{T_{f}} \cdot \nm{T_{f_{\nu}}^{n} - T_{f}^{n}}_{2}
\\ & \leq \frac{B^n}{\sqrt{\nu}} \nm{ T_{f_{\nu}} - T_{f}}_{2} + \frac{1}{\sqrt{\nu}} \nm{f}_{\infty} \cdot \nm{T_{f_{\nu}}^{n} - T_{f}^{n}}_{2}.
\end{flalign*}
This proves our claim. This also implies that
$$\lim_{\nu \rightarrow \infty} \frac{1}{\nu} \Tr(T_{f_{\nu}}^{2n} - T_{f}^{2n}) = 0$$
for any integer $n$, using properties of the Hilbert-Schmidt norm. Furthermore by Lemma \ref{tracexn} 
$$\lim_{\nu \rightarrow \infty} \frac{1}{\nu} \Tr(T_{f}^{2n}) = \int_{\D} f(z)^{2n} d \iota (z),$$
and the result follows writing
\begin{flalign*}
& \lim_{\nu \rightarrow \infty} |\frac{1}{\nu} \Tr(( (\nu - 1)R_{\nu}^*(f_{\nu}))^{2n}) - \int_{\D} f(z)^{2n} d \iota(z)|
\\ & \leq \lim_{\nu \rightarrow \infty} |\frac{1}{\nu} \Tr(( (\nu - 1)R_{\nu}^*(f_{\nu}))^{2n}) - \Tr(( (\nu - 1)R_{\nu}^*(f))^{2n})|
\\ & + \lim_{\nu \rightarrow \infty} |\Tr(( (\nu - 1)R_{\nu}^*(f))^{2n}) - \int_{\D} f(z)^{2n} d \iota(z)| = 0
\end{flalign*}
\end{proof}

We get the following theorem.

\begin{theorem}
\label{funccalcx2n}
For any $f \in C_c^{\infty}(\D)$ and integer $n$
$$ \lim_{\nu \rightarrow \infty} \frac{1}{\nu} \Tr(\mathcal{T}_{\mu,k}^{\nu}(R_{\mu}^*(f))^{2n}) =  \int_{\D} E_{\mu,k}(f)^{2n} d \iota (z).$$
\end{theorem}

\begin{proof}
We let
$$g_{\nu} \coloneqq ((\mu + \nu + 2k - 1)B_{\mu + \nu + 2k})^{-1} E_{\mu,k}^{\nu}(f).$$
By Proposition \ref{Elimit}, its proof and Lemma \ref{boundcptfctns} there is some constant $R$ such that if $\nu$ big
$$\nm{((\mu + \nu + 2k - 1)B_{\mu + \nu + 2k})^{-1} E_{\mu,k}^{\nu}(f)}_{\infty} < R.$$
Furthermore, by Proposition \ref{ourinverseberezin}
$$ \lim_{\nu \rightarrow \infty} \nm{g_{\nu} - E_{\mu,k}(f)}_2 = 0,$$
and $E_{\mu,k}(f) \in L^2(\D, d \iota) \cap L^{\infty}(\D, d \iota)$ by Definition \ref{berezindef} and the discussion following it. Hence we can use Lemma \ref{convx2n}, which proves our theorem.
\end{proof}

We get a corollary.

\begin{corollary}
\label{Husimiremark2}
Recalling Proposition \ref{Husimiremark1}, for any $A \in S^1(\Hc_{\mu})$ and integer $n$ we get
$$\lim_{\nu \rightarrow \infty} \frac{1}{\nu} \Tr(\mathcal{T}_{\mu,k}^{\nu}(A)^{2n}) =  \int_{\D} H_{\mu}^k(A)(z)^{2n} d \iota (z).$$
\end{corollary}

\begin{proof}
We remark that $R_{\mu}^*(L^1(\D, d \iota))$, and thus $R_{\mu}^*(C_c^{\infty}(\D))$, are dense in $S^1(\Hc_{\mu})$ in the $\nm{\cdot}_{1}$-norm. This can be seen as $R_{\mu}$ is injective on $B(\Hc_{\mu}) = S^1(\Hc_{\mu})^*$. Now let $A \in S^1(\Hc_{\mu})$ and $\{A_l\}_l$ be a sequence in $R_{\mu}^*(C_c^{\infty}(\D))$ such that
$$ \lim_{l \rightarrow \infty} \nm{A - A_l}_{1} = 0.$$
We claim this implies
$$ \lim_{l \rightarrow \infty} \frac{1}{\nu} \Tr( \mathcal{T}_{\mu,k}^{\nu}(A)^{n} - \mathcal{T}_{\mu,k}^{\nu}(A_l)^{n}) = 0,$$
with rate of convergence not depending on $\nu$. We prove a stronger claim, namely
$$ \lim_{l \rightarrow \infty} \frac{1}{\nu} \nm{ \mathcal{T}_{\mu,k}^{\nu}(A)^{n} - \mathcal{T}_{\mu,k}^{\nu}(A_l)^{n}}_{1} = 0,$$
with rate of convergence bounded by $R_n \nm{A - A_l}_{1}$ for some constant $R_n$ dependent on $n$. We prove this by induction. The case $n = 0$ is trivial and $n=1$ is clear from Proposition \ref{qtchannel}. We do the induction step
\begin{flalign*}
& \frac{1}{\nu} \nm{ \mathcal{T}_{\mu,k}^{\nu}(A)^{n+1} - \mathcal{T}_{\mu,k}^{\nu}(A_l)^{n+1}}_{1}
\\ & = \frac{1}{\nu} \nm{ \mathcal{T}_{\mu,k}^{\nu}(A)(\mathcal{T}_{\mu,k}^{\nu}(A)^n - \mathcal{T}_{\mu,k}^{\nu}(A_l)^n)}_{1}
\\ & + \frac{1}{\nu} \nm{ (\mathcal{T}_{\mu,k}^{\nu}(A) - \mathcal{T}_{\mu,k}^{\nu}(A_l)) \mathcal{T}_{\mu,k}^{\nu}(A_l)^n}_{1}
\\ & \leq \frac{1}{\nu} \nm{\mathcal{T}_{\mu,k}^{\nu}(A)^n - \mathcal{T}_{\mu,k}^{\nu}(A_l)^n}_{1} \nm{\mathcal{T}_{\mu,k}^{\nu}(A)}
\\ & + \frac{1}{\nu} \nm{\mathcal{T}_{\mu,k}^{\nu}(A) - \mathcal{T}_{\mu,k}^{\nu}(A_l)}_{1}  \nm{\mathcal{T}_{\mu,k}^{\nu}(A_l)}^n
\\ & \leq (R_n \nm{A} + R_1 \nm{A_l}^n) \nm{A - A_l}_{1}.
\end{flalign*}
Furthermore, by Lemma \ref{Huscont}
$$ \lim_{l \rightarrow \infty} \nm{H_{\mu}^k(A) - H_{\mu}^k(A_l)}_{\infty} = 0$$
and
$$ \lim_{l \rightarrow \infty} \nm{H_{\mu}^k(A) - H_{\mu}^k(A_l)}_{1} = 0,$$
and thus for any integer $n$
$$ \lim_{l \rightarrow \infty} \int_{\D} H_{\mu}^k(A)^{2n} d \iota (z) - \int_{\D} H_{\mu}^k(A_l)^{2n} d \iota (z) = 0. $$
Writing
\begin{flalign*}
& |\frac{1}{\nu} \Tr(\mathcal{T}_{\mu,k}^{\nu}(A)^{2n}) -  \int_{\D} H_{\mu}^k(A)^{2n} d \iota (z)|
\\ & \leq |\frac{1}{\nu} \Tr(\mathcal{T}_{\mu,k}^{\nu}(A)^{2n}) -  \frac{1}{\nu} \Tr(\mathcal{T}_{\mu,k}^{\nu}(A_l)^{2n}) |
\\ & + |\frac{1}{\nu} \Tr(\mathcal{T}_{\mu,k}^{\nu}(A_l)^{2n}) -  \int_{\D} H_{\mu}^k(A_l)^{2n} d \iota (z)|
\\ & + |\int_{\D} H_{\mu}^k(A_l)^{2n} d \iota (z) -  \int_{\D} H_{\mu}^k(A)^{2n} d \iota (z)|
\end{flalign*}
we prove our statement.
\end{proof}

We want to prove Theorem \ref{funccalcx2n} for any integer. We first make an elementary observation. There exists a Taylor expansion
$$ \sqrt{1 - x} = \sum_{i=0}^{\infty} \frac{(-\frac{1}{2})_i}{i!} x^i,$$
where this equality holds for $x \in [-1,1]$. In particular the series $\sum_{i=0}^{\infty} \frac{(-\frac{1}{2})_i}{i!}$ is absolutely convergent as
$$ |\frac{(-\frac{1}{2})_i}{i!}| = |\frac{\Gamma(- \frac{1}{2} + i)}{\Gamma( - \frac{1}{2}) \Gamma(i+1)}| \leq C|\frac{1}{\Gamma(- \frac{1}{2})} \cdot \frac{\Gamma(i)i^{-\frac{1}{2}}}{i \Gamma(i)}|=  C|\frac{1}{\Gamma(\frac{1}{2}) i^{\frac{3}{2}}}|$$
for some constant $C$. By Abel's theorem
$$\sum_{i=0}^{\infty} \frac{(-\frac{1}{2})_i}{i!} = \sqrt{1-1} = 0,$$
so
$$1 = - \sum_{i=1}^{\infty} \frac{(-\frac{1}{2})_i}{i!} = \sum_{i=1}^{\infty} \frac{ (\frac{1}{2})_{i-1}}{2 (i!)}.$$
Hence for $x \in [0,1]$
\begin{equation}
\label{taylorsqrt}
x = \sqrt{1 - (1 - x^2)} = \sum_{i=0}^{\infty} \frac{(-\frac{1}{2})_i}{i!} (1 - x^2)^i = \sum_{i=1}^{\infty} \frac{ (\frac{1}{2})_{i-1}}{2 (i!)} (1 - (1 - x^2)^i).
\end{equation}
We prove a lemma.

\begin{lemma}
\label{traceRmun}
For any $\nu$, integer $n$, and $A \in S^1(\Hc_{\mu})$ such that $A \geq 0$ and $\Tr(A) = 1$
\begin{flalign*}
& \frac{1}{\nu} \Tr(\mathcal{T}_{\mu,k}^{\nu}(A)^{n}) = \sum_{i=1}^{\infty} \frac{ (\frac{1}{2})_{i-1}}{2 (i!)} \frac{1}{\nu} \Tr( 1 - (1 - \mathcal{T}_{\mu,k}^{\nu}(A)^{2n})^i).
\end{flalign*}
\end{lemma}

\begin{proof}
By Lemma \ref{Tquantumchannel} the operator $\mathcal{T}_{\mu,k}^{\nu}(A)$ is positive and trace-class, so
$$ \mathcal{T}_{\mu,k}^{\nu}(A) = \sum_{j} \lambda_j(\nu) e_j(\nu) \otimes e_j(\nu)^*,$$
where the $e_j(\nu)$ are the orthonormal eigenvectors for the eigenvalues $\lambda_j(\nu) \geq 0$. Also $\nm{A} \leq \nm{A}_{1} \leq 1$, so by Proposition \ref{qtchannel}
$$ \nm{\mathcal{T}_{\mu,k}^{\nu}(A)} \leq 1$$
and $0 \leq \lambda_j(\nu) \leq 1$ for all $j$. Thus
\begin{flalign*}
& \frac{1}{\nu} \Tr(\mathcal{T}_{\mu,k}^{\nu}(A)^{n}) = \frac{1}{\nu} \sum_j \lambda_j(\nu)^n = \frac{1}{\nu} \sum_j \sum_{i=1}^{\infty} \frac{ (\frac{1}{2})_{i-1}}{2 (i!)} (1 - (1 - \lambda_j(\nu)^{2n})^i)
\\ & = \sum_{i=1}^{\infty} \frac{ (\frac{1}{2})_{i-1}}{2 (i!)} \frac{1}{\nu} \sum_j (1 - (1 - \lambda_j(\nu)^{2n})^i)
\\ & = \sum_{i=1}^{\infty} \frac{ (\frac{1}{2})_{i-1}}{2 (i!)} \frac{1}{\nu} \Tr( 1 - (1 - \mathcal{T}_{\mu,k}^{\nu}(A)^{2n})^i).
\end{flalign*}
\end{proof}

We now prove a theorem.

\begin{theorem}
\label{funccalcxn}
Let $A \in S^1(\Hc_{\mu})$ such that $A \geq 0$ and $\Tr(A) = 1$. Then
$$ \lim_{\nu \rightarrow \infty} \frac{1}{\nu} \Tr(\mathcal{T}_{\mu,k}^{\nu}(A)^{n}) =  \int_{\D} H_{\mu}^k(A)(z)^{n} d \iota (z)$$
for any integer $n$. Furthermore, for $f \in L^1(\D, d \iota)$ such that $R_{\mu}^*(f) \geq 0$ and $\Tr(R_{\mu}^*(f)) = 1$ we get
$$ \lim_{\nu \rightarrow \infty} \frac{1}{\nu} \Tr(\mathcal{T}_{\mu,k}^{\nu}(R_{\mu}^*(f))^{n}) =  \int_{\D} E_{\mu,k}(f)(z)^{n} d \iota (z)$$
for any integer $n$.
\end{theorem}

\begin{proof}
Note that for any $z \in \D$ we have $0 \leq H_{\mu}^k(A)(z) \leq 1$. By Equation (\ref{taylorsqrt})
\begin{equation}
\label{intsumsquare}
H_{\mu}^k(A)^n = \sum_{i=1}^{\infty} \frac{ (\frac{1}{2})_{i-1}}{2 (i!)} (1 - (1 - H_{\mu}^k(A)^{2n})^{i}),
\end{equation}
implying
\begin{equation}
\label{Emunusum}
\int_{\D} H_{\mu}^k(A)(z)^n d \iota(z) = \sum_{i=1}^{\infty} \frac{ (\frac{1}{2})_{i-1}}{2 (i!)} \int_{\D} (1 - (1 - H_{\mu}^k(A)(z)^{2n})^{i}) d \iota(z).
\end{equation}
This is finite for $n=1$ by Lemma \ref{traceclassint}, so by the boundedness of $H_{\mu}^k(A)$ it is finite for all $n$. Now note that by Lemma \ref{traceRmun} for $n=1$, Equation (\ref{intsumsquare}), and Remark \ref{Husimiremark2}
\begin{flalign*}
& \lim_{\nu \rightarrow \infty} \sum_{i=1}^{\infty} \frac{ (\frac{1}{2})_{i-1}}{2 (i!)} \frac{1}{\nu} \Tr( 1 - (1 - \mathcal{T}_{\mu,k}^{\nu}(A)^{2})^i) = \lim_{\nu \rightarrow \infty} \frac{1}{\nu} \Tr(\mathcal{T}_{\mu,k}^{\nu}(A))
\\ & = \lim_{\nu \rightarrow \infty} \frac{\mu + \nu + 2k - 1}{\nu(\mu-1)} \Tr(A) = \int_{\D} H_{\mu}^k(A)(z) d \iota(z)
\\ & = \sum_{i=1}^{\infty} \frac{ (\frac{1}{2})_{i-1}}{2 (i!)} \int_{\D} 1 - (1 - H_{\mu}^k(A)^{2})^{i} d \iota(z)
\\ & = \sum_{i=1}^{\infty} \frac{ (\frac{1}{2})_{i-1}}{2 (i!)} \lim_{\nu \rightarrow \infty} \frac{1}{\nu} \Tr( 1 - (1 - \mathcal{T}_{\mu,k}^{\nu}(A)^{2})^i).
\end{flalign*}
This implies that for any integer $I$
\begin{flalign*}
& \sum_{i=I+1}^{\infty} \lim_{\nu \rightarrow \infty} \frac{ (\frac{1}{2})_{i-1}}{2 (i!)} \frac{1}{\nu} \Tr( 1 - (1 - \mathcal{T}_{\mu,k}^{\nu}(A)^{2})^i)
\\ & = \lim_{\nu \rightarrow \infty} \sum_{i=I+1}^{\infty} \frac{ (\frac{1}{2})_{i-1}}{2 (i!)} \frac{1}{\nu} \Tr( 1 - (1 - \mathcal{T}_{\mu,k}^{\nu}(A)^{2})^i).
\end{flalign*}
This means that for $\epsilon > 0$ we can choose $I$ such that
$$ \sum_{i=I+1}^{\infty} \frac{ (\frac{1}{2})_{i-1}}{2 (i!)} \lim_{\nu \rightarrow \infty} \frac{1}{\nu} \Tr( 1 - (1 - \mathcal{T}_{\mu,k}^{\nu}(A)^{2})^i) < \epsilon$$
and
$$\sum_{i=I+1}^{\infty} \frac{ (\frac{1}{2})_{i-1}}{2 (i!)} \int_{\D} 1 - (1 - H_{\mu}^k(A)(z)^{2})^i d \iota(z) < \epsilon,$$
and $N_0$ such that $\nu \geq N_0$ implies
$$ \sum_{i=I+1}^{\infty} \frac{ (\frac{1}{2})_{i-1}}{2 (i!)} \frac{1}{\nu} \Tr( 1 - (1 - \mathcal{T}_{\mu,k}^{\nu}(A)^{2})^i) < 2 \epsilon.$$
Note that all the terms in these sums are positive. Now the inequality
$$ 0 \leq 1 - (1 - x^{2n})^i \leq 1 - (1 - x^2)^i$$
for $0 \leq x \leq 1$ implies for any $n$
$$ \sum_{i=I+1}^{\infty} \frac{ (\frac{1}{2})_{i-1}}{2 (i!)} \int_{\D} 1 - (1 - H_{\mu}^k(A)(z)^{2n})^i d \iota(z) < \epsilon,$$
and for any $n$ and $\nu \geq N_0$
$$ \sum_{i=I+1}^{\infty} \frac{ (\frac{1}{2})_{i-1}}{2 (i!)} \frac{1}{\nu} \Tr( 1 - (1 - \mathcal{T}_{\mu,k}^{\nu}(A)^{2n})^i) < 2 \epsilon.$$
Now we choose $N \geq N_0$ such that for $\nu \geq N$
\begin{flalign*}
& | \sum_{i=1}^{I} \frac{ (\frac{1}{2})_{i-1}}{2 (i!)} \frac{1}{\nu} \Tr( 1 - (1 - \mathcal{T}_{\mu,k}^{\nu}(A)^{2n})^i)
\\ & - \sum_{i=1}^I \frac{ (\frac{1}{2})_{i-1}}{2 (i!)} \int_{\D} (1 - (1 - H_{\mu}^k(A)(z)^{2n})^i) d \iota(z)| < \epsilon,
\end{flalign*}
which is possible by Remark \ref{Husimiremark2}. Now by Lemma \ref{traceRmun} and Equation \ref{Emunusum}, for $\nu \geq N$ we get
\begin{flalign*}
& | \frac{1}{\nu} \Tr( \mathcal{T}_{\mu,k}^{\nu}(A)^n) - \int_{\D} H_{\mu}^k(A)(z)^n d \iota(z)|
\\ & \leq | \frac{1}{\nu} \Tr( \mathcal{T}_{\mu,k}^{\nu}(A)^n - \sum_{i=1}^{I} \frac{ (\frac{1}{2})_{i-1}}{2 (i!)} \frac{1}{\nu} \Tr( 1 - (1 - \mathcal{T}_{\mu,k}^{\nu}(A)^{2n})^i)|
\\ & + |\sum_{i=1}^{I} \frac{ (\frac{1}{2})_{i-1}}{2 (i!)} \frac{1}{\nu} \Tr( 1 - (1 - \mathcal{T}_{\mu,k}^{\nu}(A)^{2n})^i)
\\ & - \sum_{i=1}^I \frac{ (\frac{1}{2})_{i-1}}{2 (i!)} \int_{\D} (1 - (1 - H_{\mu}^k(A)(z)^{2n})^i) d \iota(z)|
\\ & + |\sum_{i=1}^I \frac{ (\frac{1}{2})_{i-1}}{2 (i!)} \int_{\D} (1 - (1 - 
H_{\mu}^k(A)(z)^{2n})^i) d \iota(z) - \int_{\D} H_{\mu}^k(A)(z)^n d \iota(z)|
\\ & \leq |\sum_{i=I+1}^{\infty} \frac{ (\frac{1}{2})_{i-1}}{2 (i!)} \frac{1}{\nu} \Tr( 1 - (1 - \mathcal{T}_{\mu,k}^{\nu}(A)^{2n})^i)|
\\ & + \epsilon + |\sum_{i=I+1}^{\infty} \frac{ (\frac{1}{2})_{i-1}}{2 (i!)} \int_{\D} (1 - (1 - H_{\mu}^k(A)(z)^{2n})^i) d \iota(z)|
< 4 \epsilon.
\end{flalign*}
This proves our theorem as the second part follows directly from the first.
\end{proof}

We now calculate the limit of the functional calculus for any $\psi \in C([0,1])$. However, $\psi(\mathcal{T}_{\mu,k}^{\nu}(A))$ will not always be a trace-class operator for any function; a trivial example is $\psi = 1$ which will result in the identity operator. We show for which functions we might expect $\psi(\mathcal{T}_{\mu,k}^{\nu}(A))$ is trace-class, and some examples for which it is not.

\begin{lemma}
For $\psi \in C([0,1]) \cap O(x) \coloneqq \{ f \in C([0,1]) \mid \exists_{c \in \R} \ |f(x)| \leq c|x| \}$ the operator $\psi(\mathcal{T}_{\mu,k}^{\nu}(A))$ is trace-class for a positive trace-class operator $A$. If $\psi(x) = x^p$ where $p < 1$, there is some $A \in S^1(\Hc_{\mu})$ such that $\psi(A)$ is no longer trace-class.
\end{lemma}

\begin{proof}
Let $A$ be trace-class and $\psi \in C([0,1]) \cap O(x)$, then $\mathcal{T}_{\mu,k}^{\nu}(A)$ also is by Proposition \ref{Tquantumchannel}. Thus
$$ \mathcal{T}_{\mu,k}^{\nu}(A) = \sum_{i=1}^{\infty} \lambda_i (e_i \otimes e_i^*)$$
for some orthonormal basis $\{e_i\}_i$ and $\sum_{i=1}^{\infty} |\lambda_i| < \infty$. Then
$$\Tr(\psi(\mathcal{T}_{\mu,k}^{\nu}(A)) = \sum_{i=1}^{\infty} \psi(\lambda_i) \leq \sum_{i=1}^{\infty} c \lambda_i < \infty.$$

Now let $p < 1$ and define the series $\{ \lambda_i \}_i$ positive and decreasing such that $\sum_{i=1}^{\infty} \lambda_i$ is (absolutely) convergent, but $\sum_{i=1}^{\infty} \lambda_i^p$ diverges. Define
$$ A =  \sum_{i=1}^{\infty} \lambda_i (e_i \otimes e_i^*)$$
for the orthonormal basis $\{e_i = (\frac{(\mu)_i}{i!})^{\frac{1}{2}} z^i \}_i$ of $\Hc_{\mu}$. Then all the operators $\mathcal{T}_{\mu,0}^{\nu}(e_i \otimes e_i^*)$ are simultaneously diagonalizable with eigenvectors $\{ z^n \}_n$, where furthermore the eigenvalue of $z^n$ is $0$ whenever $i > n$. Hence the operator $\mathcal{T}_{\mu,k}^{\nu}(A)$ is diagonalizable with eigenvectors $\{z^n\}_n$ and
\begin{flalign*}
& \langle \mathcal{T}_{\mu,0}^{\nu}(A)(z^n), z^n \rangle = \sum_{i=0}^n \lambda_i \langle P_k( \mathcal{T}_{\mu,k}^{\nu}(e_i \otimes e_i^*) z^n, z^n \rangle
\\ & \geq \lambda_n \sum_{i=0}^{n} \langle \mathcal{T}_{\mu,k}^{\nu}(e_i \otimes e_i^*) z^n, z^n \rangle = \lambda_n \sum_{i=0}^{\infty} \langle P_0( (e_i \otimes e_i^*) \otimes I_{\nu}) P_0^* z^n, z^n \rangle
\\ & = \lambda_n \langle z^n, z^n \rangle.
\end{flalign*}
Thus for the trace
$$ \Tr(\mathcal{T}_{\mu,k}^{\nu}(A)^p) \geq \sum_{i=0}^{\infty} \lambda_i^p = \infty.$$
\end{proof}

We now find the functional calculus. We consider the functional calculus on $C([0,1])$.

\begin{theorem}
For $\psi \in C^1([0,1]) $, $\psi(0) = 0$ and $A \in S^1(\Hc_{\mu})$ such that $A \geq 0$ and $\Tr(A) = 1$
$$ \lim_{\nu \rightarrow \infty} \frac{1}{\nu} \Tr(\psi(\mathcal{T}_{\mu,k}^{\nu}(A))) = \int_{\D} \psi(H_{\mu}^k(A)(z)) d \iota (z).$$
Furthermore, if $f \in L^1(\D, d \iota)$ such that $R_{\mu}^*(f) \geq 0$ and $\Tr(R_{\mu}^*(f)) = \int_{\D} f(z) d \iota(z) = 1$
$$ \lim_{\nu \rightarrow \infty} \frac{1}{\nu} \Tr(\psi(\mathcal{T}_{\mu,k}^{\nu}(R_{\mu}^*(f)))) = \int_{\D} \psi(E_{\mu,k}(f)(z)) d \iota (z).$$
\end{theorem}

\begin{proof}
First we note that $\mathcal{T}_{\mu,k}^{\nu}(A)$ has eigenvalues bounded by $1$, so the functional calculus $\psi(\mathcal{T}_{\mu,k}^{\nu}(A))$ is well-defined. Also note that $\psi(H_{\mu}^k(A)(z))$ is always well-defined. Now we note that $x \mapsto \frac{\psi(x)}{x}$ is in $C([0,1])$, if we define it to be
$$ \lim_{x \rightarrow 0^+} \frac{\psi(x)}{x} = \frac{\psi'(0)}{1} = \psi'(0)$$
at $0$. Hence it is the limit of a sequence of polynomials $\{P_n\}_n$. We define $Q_n(x) = x P_n(x)$ and see $| \frac{\psi(x)}{x} - P_n(x)| < \epsilon_n$ implies
$$ | \psi(x) - Q_n(x) | < \epsilon_n x.$$
Now we note that if $\{ \lambda_i \}_i$ are the eigenvalues of $\mathcal{T}_{\mu,k}^{\nu}(A)$ then
\begin{flalign*}
& \frac{1}{\nu} |\Tr(\psi(\mathcal{T}_{\mu,k}^{\nu}(A))) - \Tr(Q_n(\mathcal{T}_{\mu,k}^{\nu}(A)))|
\\ & \leq \frac{1}{\nu} \sum_{i=1}^{\infty} | \psi(\lambda_i) - Q_n(\lambda_i) | \leq \frac{\epsilon_n}{\nu} \sum_{i=1}^{\infty} \lambda_i = \frac{\epsilon_n}{\nu} \Tr(\mathcal{T}_{\mu,k}^{\nu}(A))
\\ & = \epsilon_n \frac{\mu + \nu + 2k -1}{\nu(\mu -1)}.
\end{flalign*}
Also using Lemma \ref{traceclassint}
\begin{flalign*}
& |\int_{\D} \psi(H_{\mu}^k(A)(z)) d \iota (z) - \int_{\D} Q_n(H_{\mu}^k(A)(z)) d \iota (z)| 
\\ & \leq \int_{\D} |\psi(H_{\mu}^k(A)(z)) - Q_n(H_{\mu}^k(A)(z))| d \iota (z)
\\ & \leq \int_{\D} \epsilon_n H_{\mu}^k(A)(z) d \iota(z) = \frac{\epsilon_n}{\mu-1}.
\end{flalign*}
Finally, by Remark \ref{Husimiremark2} for each $n$
$$ \lim_{\nu \rightarrow \infty} \frac{1}{\nu} \Tr(Q_n(\mathcal{T}_{\mu,k}^{\nu}(A))) = \int_{\D} Q_n(H_{\mu}^k(A)(z)) d \iota (z).$$
Now if $\epsilon > 0$, we can first choose $\epsilon_n$ such that for any $\nu$
$$\epsilon_n \frac{\mu + \nu + 2k -1}{\nu(\mu -1)} < \epsilon,$$
and $\frac{\epsilon_n}{\mu-1} < \epsilon$, and then we can choose $N$ big such that for $\nu \geq N$
$$ | \frac{1}{\nu} \Tr(Q_n(\mathcal{T}_{\mu,k}^{\nu}(A))) - \int_{\D} Q_n(H_{\mu}^k(A)(z)) d \iota (z)| < \epsilon.$$
It follows that for $\nu \geq N$
\begin{flalign*}
& | \frac{1}{\nu} \Tr(\psi(\mathcal{T}_{\mu,k}^{\nu}(A))) - \int_{\D} \psi(H_{\mu}^k(A)(z)) d \iota (z) |
\\ & \leq | \frac{1}{\nu} \Tr(\psi(\mathcal{T}_{\mu,k}^{\nu}(A))) - \frac{1}{\nu} \Tr(Q_n(\mathcal{T}_{\mu,k}^{\nu}(A)))|
\\ & + |\frac{1}{\nu} \Tr(Q_n(\mathcal{T}_{\mu,k}^{\nu}(A))) - \int_{\D} Q_n(H_{\mu}^k(A)(z)) d \iota (z)|
\\ & + |\int_{\D} Q_n(H_{\mu}^k(A)(z)) d \iota (z) - \int_{\D} \psi(H_{\mu}^k(A)(z)) d \iota (z)|
< 3 \epsilon
\end{flalign*}
for $\nu \geq N$, and we obtain the first part of the theorem. The second part follows from the first.
\end{proof}

We note that implicit in the proof is that $\psi \in C^1([0,1])$ and $\psi(x) = 0$ implies $\psi \in O(|x|)$.

There are several natural questions related to this result. One of these is the question of majorization. It is proved by Lieb and Solovej \cite{liebsolBloch} that for the $SU(2)$-equivariant quantum channels corresponding to the highest weight, the operator $\mathcal{T}_{\mu,0}^{\nu}(A)$, where $A \geq 0$ and $\Tr(A) = 1$, is majorized for a projection onto the highest weight vector. This was later generalized to some representations of $SU(n)$ \cite{liebsolSymm}. For the $SU(1,1)$ case, I believe a similar property holds considering the lowest weight component. Kulikov's result \cite{kuli} can be formulated as this being true in the limit, and computer simulations suggest majorization occurs, but I have not been able to prove this. 

There are also more questions about the classification of equivariant maps and the limit in different Banach spaces \cite{soletal}. The $SU(2)$-invariant channels have the structure of a simplex \cite{alnu}, and Aschieri, Ruba and Solovej concluded that the limit of the trace of the functional calculus of extremal channels can be written as extremal points of $SU(2)$-equivariant POVMs. We wonder if a similar result holds for $SU(1,1)$-equivariant channels. Furthermore, in \cite{soletal} limits and inequalities in different Banach spaces were considered regarding the difference of the Toeplitz operators $\nm{T_f T_g - T_{fg}}$ in different Banach spaces, and bounds were proved for expansions. The bounds were dependent on the derivative(s) of the functions. Explicit bounds were also given for the convergence. Perhaps similar statements can be obtained for $SU(1,1)$.

It also seems that this limiting procedure of equivariant channels can be put in a more general context of holomorphic discrete series representations of Hermitian Lie groups \cite{pezh, zhangCon}. Genkai Zhang and I are planning to study this in a future publication. These quantum channels might even be studied in the frame of general Toeplitz-quantization \cite{alieng, engl}.


\begin{thebibliography}{20}
\bibitem{alieng} S.T. Ali, M. Englis:
\emph{Quantization methods: a guide for physicists and analysts},
Reviews in Mathematical Physics {\bf 17} (2005) 391--490.

\bibitem{alnu} M. Al Nuwairan:
\emph{The extreme points of SU(2)-irreducibly covariant channels},
International Journal of Mathematics {\bf 25} (2014) 1450048.

\bibitem{soletal} T. Aschieri, B. Ruba, B., J.P. Solovej:
\emph{$SU(2)$-equivariant quantum channels: semiclassical analysis},
arXiv:2404.04133 (2024).

\bibitem{andaskroy} R. Askey, G.E. Andrews, R. Roy:
\emph{Special Functions},
Encyclopedia of Mathematics and its Applications, 71, Cambridge University Press , Cambridge. (1999).

\bibitem{foll} G.B. Folland:
\emph{Real analysis: modern techniques and their applications}
Pure and Applied Mathematics, John Wiley \& Sons, Inc. New York et al. (1999).

\bibitem{BMSCMP} M. Bordemann, E. Meinrenken, M. Schlichenmaier:
\emph{Toeplitz quantization of Kähler manifolds and $\mathrm{gl}(N)$, $N \rightarrow \infty$ limits},
Communications in Mathematical Physics {\bf 165} (1994) 281--296.

\bibitem{engl} M. Engliš:
\emph{Weighted Bergman Kernels and Quantization},
Communications in Mathematical Physics {\bf 227} (2002)  211–241.

\bibitem{frank} R.L. Frank:
\emph{Sharp inequalities for coherent states and their optimizers},
Advanced Nonlinear Studies { \bf 23} (2023) 20220050.

\bibitem{haap} E. Haapasalo:
\emph{Compatibility of covariant quantum channels with emphasis on Weyl symmetry},
Annales Henri Poincaré {\bf 20} (2019) 3163--3195.

\bibitem{haas} R. van Haastrecht:
\emph{Limit formulas for the trace of the functional calculus of quantum channels for $SU(2)$}
Journal of Lie Theory {\bf 34} (2024), to be published.

\bibitem{hari} Harish-Chandra:
\emph{Representations of semisimple Lie groups VI: Integrable and square-integrable representations},
American Journal of Mathematics {\bf 78} (1956) 564--628.

\bibitem{helgDS} S. Helgason:
\emph{Differential Geometry, Lie Groups, and Symmetric Spaces},
Academic Press, New York et al. (1979).

\bibitem{helgGGA} S. Helgason:
\emph{Groups and Geometric Analysis: Integral Geometry, Invariant Differential Operators, and Spherical Functions},
Pure and applied mathematics, Academic Press, Orlando et al. (1984).

\bibitem{knapp} A.W. Knapp:
\emph{Representation theory of semisimple groups: an overview based on examples},
Princeton Mathematical Series, Princeton University Press, Princeton. (1986). 

\bibitem{kuli} A. Kulikov:
\emph{Functionals with extrema at reproducing kernels},
Geometric and Functional Analysis {\bf 32} (2022) 938--949.

\bibitem{lieb} E.H. Lieb:
\emph{Proof of an entropy conjecture of Wehrl},
Communications in Mathematical Physics {\bf 62} (1978) 35--41.

\bibitem{liebsolBloch} E.H. Lieb, J.P. Solovej:
\emph{Proof of an entropy conjecture for Bloch coherent spin states and its generalizations},
Acta Mathematica {\bf 212} (2014) 379--398.

\bibitem{liebsolSU11} E.H. Lieb, J.P. Solovej:
\emph{Wehrl-type coherent state entropy inequalities for $ SU (1, 1) $ and its $ AX+ B $ subgroup},
in: Partial differential equations, spectral theory, and mathematical physics. The Ari Laptev anniversary volume, P. Exner, R.L. Frank, F. Gesztesy, H. Holden, T. Weidl (eds.), European Mathematical Society (EMS), Berlin. (2021) 301--314.

\bibitem{liebsolSymm} E.H. Lieb, J.P. Solovej:
\emph{Proof of the Wehrl-type Entropy Conjecture for Symmetric $SU(N)$ Coherent States},
Communications in Mathematical Physics {\bf 348} (2016) 567--578. 
arXiv preprint arXiv:1906.00223 (2019).
Acta Mathematica {\bf 212} (2014) 379--398.

\bibitem{peet} J. Peetre:
\emph{Hankel forms of arbitrary weight over a symmetric domain via the transvectants},
The Rocky Mountain Journal of Mathematics {\bf 24} (1994) 1065--1085.

\bibitem{pezh} L. Peng, G. Zhang:
\emph{Tensor products of holomorphic representations and bilinear differential operators},
Journal of Functional Analysis {\bf 210} (2004) 171--192.

\bibitem{petz} D. Petz:
\emph{Quantum information theory and quantum statistics},
Theoretical and Mathematical Physics, Springer, Berlin et al. (2008).

\bibitem{repk} J. Repka:
\emph{Tensor products of unitary representations of $SL(2,\R)$},
Bulletin of the American Mathematical Society {\bf 82} (1976) 930--932.

\bibitem{sugi} A. Sugita:
\emph{Proof of the generalized Lieb-Wehrl conjecture for integer indices larger than one},
Journal of Physics A: Mathematical and General {\bf 35} (2002) L621--L626.

\bibitem{untup} A. Unterberger, H. Upmeier:
\emph{The Berezin transform and invariant differential operators}, Communications in Mathematical Physics {\bf 164} (1994) 563--597.

\bibitem{wehrl} A. Wehrl:
\emph{On the relation between classical and quantum-mechanical entropy},
Reports on Mathematical Physics {\bf 16} (1979) 353--358.

\bibitem{zag} D. Zagier:
\emph{Modular forms and differential operators},
Proceedings Mathematical Sciences {\bf 104} (1994) 57--75.

\bibitem{zhangbound} G. Zhang:
\emph{Berezin Transform on Real Bounded Symmetric Domains},
Transactions of the American Mathematical Society {\bf 353} (2001) 3769-–3787.

\bibitem{zhangCon} G. Zhang:
\emph{Wehrl-type inequalities for Bergman spaces on domains in $\C^d$ and completely positive maps},
in: The Bergman Kernel and Related Topics, Hayama Symposium on SCV XXIII at Kanagawa, Japan, July 2022, K. Hirachi, T. Ohsawa, S. Takayama and J. Kamimoto (eds.), Springer, Singapore. (2024) 343--355.

\bibitem{zhu} K. Zhu:
K. (2007). 
\emph{Operator theory in function spaces}
Mathematical Surveys and Monographs {\bf 138}, American Mathematical Society, Providence. (2007).
\end{thebibliography}
\end{document}